\documentclass[12pt]{amsart}

\usepackage{amssymb,bm}
\usepackage[dvipdfmx]{graphicx}
\usepackage{ytableau}
\usepackage[all]{xy} 
\usepackage{hyperref}

\numberwithin{equation}{section}

\newtheorem{theorem}{Theorem}[section]
\newtheorem{proposition}[theorem]{Proposition}
\newtheorem{lemma}[theorem]{Lemma}
\newtheorem{corollary}[theorem]{Corollary}

\theoremstyle{definition}

\newtheorem{example}[theorem]{Example}

\newtheorem{question}[theorem]{Question}

\newtheorem{remark}[theorem]{Remark}

\newcommand{\ZZ}{ \ensuremath{\mathbb{Z}}}

\newcommand{\reg}{\mathop{\mathrm{reg}}}

\newcommand{\Tor}{\ensuremath{\mathrm{Tor}}\hspace{1pt}}
\newcommand{\Hom}{\ensuremath{\mathrm{Hom}}\hspace{1pt}}
\newcommand{\Ext}{\ensuremath{\mathrm{Ext}}\hspace{1pt}}

\newcommand{\pdim}{\operatorname{pdim}}

\newcommand{\ee}{\mathbf{e}}
\newcommand{\eee}{\mathfrak{e}}

\def\cocoa{{\hbox{\rm C\kern-.13em o\kern-.07em C\kern-.13em o\kern-.15em A}}}

\newcommand{\bb}[1]{\mathbb{#1}}
\newcommand{\bma}{{\bm a}}
\newcommand{\bmb}{{\bm b}}
\newcommand{\bmc}{{\bm c}}
\newcommand{\bmd}{{\bm d}}
\newcommand{\bme}{{\bm e}}
\newcommand{\bt}{\boxtimes}
\newcommand{\bw}{\bigwedge}
\newcommand{\com}{\mathfrak{com}}
\newcommand{\defi}[1]{{\bf \upshape\sffamily #1}}
\newcommand{\gr}{\operatorname{gr}}
\newcommand{\grmod}{\mathfrak{grmod}}
\def\lra{\longrightarrow}
\def\lla{\longleftarrow}
\newcommand{\mc}[1]{\mathcal{#1}}
\newcommand{\mf}[1]{\mathfrak{#1}}
\newcommand{\mods}{\mathfrak{mod}}
\newcommand{\oo}{\otimes}
\newcommand{\supp}{\operatorname{supp}}
\newcommand{\ul}[1]{\underline{#1}}
\newcommand{\vecs}{\mathfrak{vec}}

\newcommand{\kk}{\Bbbk}

\newcommand{\pd}{\partial}

\newcommand{\sym}{\ensuremath{\mathfrak{S}}}
\newcommand{\lcm}{\ensuremath{\mathop{\mathrm{lcm}}}}

\begin{document}

\title[Equivariant Hochster's formula]{An equivariant Hochster's formula\\ for $\sym_n$-invariant monomial ideals}

\author{Satoshi Murai}
\address{
Satoshi Murai,
Department of Mathematics
Faculty of Education
Waseda University,
1-6-1 Nishi-Waseda, Shinjuku, Tokyo 169-8050, Japan}
\email{s-murai@waseda.jp}

\author{Claudiu Raicu}
\address{Department of Mathematics, University of Notre Dame, 255 Hurley, Notre Dame, IN~46556, USA\newline
\indent Institute of Mathematics ``Simion Stoilow'' of the Romanian Academy}
\email{craicu@nd.edu}


\begin{abstract}
Let $R=\kk[x_1,\dots,x_n]$ be a polynomial ring over a field $\kk$ and let $I\subset R$ be a monomial ideal preserved by the natural action of the symmetric group $\sym_n$ on~$R$. We give a combinatorial method to determine the $\sym_n$-module structure of $\Tor_i(I,\kk)$.
Our formula shows that $\Tor_i(I,\kk)$ is built from induced representations of tensor products of Specht modules associated to hook partitions, and their multiplicities are determined by topological Betti numbers of certain simplicial complexes. 
This result can be viewed as an $\sym_n$-equivariant analogue of Hochster's formula for Betti numbers of monomial ideals. We apply our results to determine extremal Betti numbers of $\sym_n$-invariant monomial ideals, and in particular recover formulas for their Castelnuovo--Mumford regularity and projective dimension. We also give a concrete recipe for how the Betti numbers change as we increase the number of variables, and in characteristic zero (or $>n$) we compute the $\sym_n$-invariant part of $\Tor_i(I,\kk)$ in terms of $\Tor$ groups of the unsymmetrization of $I$.
\end{abstract}

\maketitle

\section{Introduction}
\ytableausetup{centertableaux}
\ytableausetup{boxsize=0.3em}

Let $R=\kk[x_1,\dots,x_n]$ be a polynomial ring over a field $\kk$.
The study of the graded Betti numbers
$$\beta_{ij}(I)=\dim_\kk \Tor_i(I,\kk)_j$$
of a homogeneous ideal $I$ is one of the central research topics in commutative algebra. When $I$ is a monomial ideal, this is closely related to combinatorial topology, and is the subject of a vast literature \cite[Section~5]{BH}, \cite[Part~II]{HH}, \cite[Part~I]{MS-book}, \cite[Part~III]{Pe-book}. One of the most famous results on this topic is Hochster's formula \cite[\S 5.5]{BH}, which enables one to study Betti numbers of monomial ideals using combinatorics of simplicial complexes, and has seen numerous applications over the years. The goal of this paper is to develop an analogue of Hochster's formula for monomial ideals~$I$ that are invariant under the action of the symmetric group $\sym_n$ by coordinate permutations. 

In general, if $G\subseteq\operatorname{GL}_n(\kk)$ is a subgroup and $I\subseteq R$ is a $G$-invariant ideal ($g(I)=I$ for all $g\in G$), then each $\Tor_i(I,\kk)$ acquires a $G$-module structure. The representation theory of $G$ dictates the possible building blocks that make up $\Tor_i(I,\kk)$, reducing the problem of understanding $\Tor_i(I,\kk)$ to that of identifying the multiplicity of each building block. When $G=\kk^{\times n}$ is the $n$-torus, the building blocks are $1$-dimensional, given by the torus characters, and the calculation of their multiplicity (the dimension of the multigraded components of $\Tor_i(I,\kk)$) is the content of Hochster's formula. When $G=\kk^{\times n}\rtimes\sym_n$ extends the torus action by that of the symmetric group, our work will show that a natural set of building blocks arises via induction from Specht modules associated to hook partitions for smaller symmetric groups. We then identify the multiplicities of each block with topological Betti numbers of associated simplicial complexes, which yields an \defi{$\sym_n$-equivariant Hochster's formula}.

To state our results, we first introduce some notation. We let
$$P_n=\{(\lambda_1,\dots,\lambda_n) \in \ZZ^n \mid \lambda_1 \geq \cdots \geq \lambda_n \geq 0\}$$
be the set of partitions consisting of $n$ non-negative integers. For a vector $\bma=(a_1,\dots,a_n) \in \ZZ_{\geq 0}^n$,
we write $x^{\bm a}=x_1^{a_1} \cdots x_n^{a_n}$, and write $\mathrm{part}(\bm a) \in P_n$ for the unique partition which is a rearrangement of $a_1,\dots,a_n$.
For example, we have $\mathrm{part}(2,1,3,2)=(3,2,2,1)$. Throughout the text, $I$ will denote an $\sym_n$-invariant monomial ideal in $R$. For such $I$, we let 
$$P(I)=\{\lambda \in P_n \mid x^\lambda \in I\},$$
and think of $P(I)$ informally as the set of partitions in $I$. For partitions $\lambda^1,\lambda^2,\dots,\lambda^r \in P_n$, we define
\begin{equation}\label{eq:explicit-sym-invariant-I}
\textstyle 
\langle \lambda^1,\dots,\lambda^r\rangle_{\sym_n}
= \sum_{k=1}^r (\sigma(x^{\lambda^k}) \mid \sigma \in \sym_n) \subset \kk [x_1,\dots,x_n]
\end{equation}
and call it the $\sym_n$-invariant monomial ideal generated by $\lambda^1,\dots,\lambda^r$. For instance, we have
\begin{equation}\label{eq:I-411-520}
\begin{aligned}
I&=\langle (4,1,1),(5,2,0)\rangle_{\sym_3}\\&=(x_1^4x_2x_3,x_1x_2^4x_3,x_1x_2x_3^4,x_1^5x_2^2,
x_1^5x_3^2,x_1^2x_2^5,x_1^2x_3^5,x_2^5x_3^2,
x_2^2x_3^5).
\end{aligned}
\end{equation}
If $M$ is a $\ZZ^n$-graded $R$-module and $\bma\in \ZZ^n$, then $M_{\bma}$ denotes the $\bma$-th graded component of $M$. Let
$$M_{\langle \lambda \rangle} = \bigoplus_{\substack{\bma \in \ZZ^n \\ \mathrm{part}(\bma)=\lambda}} M_{\bm a}$$
for $\lambda\in P_n$. We note that $\Tor_i(I,\kk)_{\langle \lambda \rangle}$ is fixed by the $\sym_n$-action, so we have a decomposition
$$\Tor_i(I,\kk)= \bigoplus_{\lambda \in P_n} \Tor_i(I,\kk)_{\langle \lambda \rangle}$$
as $\sym_n$-modules. Therefore, we may focus on the $\sym_n$-module structure of each $\Tor_i(I,\kk)_{\langle \lambda \rangle}$.

It will be convenient to use the abbreviation 
$$(a_1^{p_1},a_2^{p_2},\dots,a_s^{p_s})=(a_1,\dots,a_1,a_2,\dots,a_2,\dots,a_s,\dots,a_s)$$
where each $a_k$ appears $p_k$ times on the right side, and to identify each partition $\mu$ with its Young diagram. For example, $(2^2,1)$ will be identified with $\ydiagram{2,2,1}$. For a partition
\begin{align}
\label{num}
\mu=(d_1^{p_1},\dots,d_s^{p_s},0^{p_{s+1}}) \in P_n,
\end{align}
where $d_1> \cdots >d_s >0$, $p_1,\dots,p_s>0$ and $p_{s+1} \geq 0$, we define
\begin{equation}\label{eq:def-p-mu}
\bm p(\mu)=(p_1-1,\dots,p_s-1)
\ \ \mbox{ and } \ \  s(\mu)=s.
\end{equation}
For a vector $\bmc=(c_1,\dots, c_s) \in \ZZ_{\geq 0}^s $ with $\bm c \leq (p_1,\dots,p_s)$,
we define
\begin{equation}\label{eq:mu-minus-c}
\mu\setminus \bmc= (d_1^{p_1-c_1},(d_1-1)^{c_1},d_2^{p_2-c_2},(d_2-1)^{c_2},\dots,d_s^{p_s-c_s},(d_s-1)^{c_s},0^{p_{s+1}}) \in P_n.
\end{equation}

\begin{example}\label{ex:pmu-mu-c}
Suppose that $\mu=(5^2,3^2,2^2)$ and $n=6$. We have $s=3$, $p_1=p_2=p_3=2$, and $p_4=0$, hence $\bm p(\mu)=(1,1,1)$. If we let $\bmc=(1,2,1)$ then
\[(5^2,3^2,2^2)\setminus (1,2,1)=(5,4,2^3,1).\]
As illustrated below, we can think of $(5,4,2^3,1)$ as being obtained from $(5^2,3^2,2^2)$ by removing one box from the fifth column, two boxes from the third column, and one box from the second column:
\ytableausetup{smalltableaux}
$$
\begin{array}{c}
\ydiagram{5,5,3,3,2,2} \\[25pt]
(5^2,3^2,2^2)
\end{array}
\hskip 30pt 
\longrightarrow
\hskip 30pt 
\begin{array}{c}
\begin{ytableau} {} &{} & {} & {} & {} \\
{} & {} & {} & {} & \none[ \times ]\\
{} & {} & \none[ \times ]\\
{} & {} & \none[ \times ]\\
{} & {} \\
{} & \none[ \times ]
\end{ytableau} \\[25pt]
(5,4,2^3,1)
\end{array}
$$
\end{example}

We let $\ee_1,\dots,\ee_s$ be the standard vectors of $\ZZ^s$ and write $\ee_F=\sum_{i \in F} {\ee}_i$ for $F \subset [s]=\{1,\dots,s\}$.
For $\bm c \leq \bm p(\mu)$, we define the simplicial complex (see Section~\ref{subsec:simpl-cpx} for some background)
\begin{equation}\label{eq:def-Del-mu-c-I}
\Delta^{\mu,\bm c}(I)=\{ F \subset [s] \mid \mu \setminus (\bm c + \ee_F) \in P(I)\},
\end{equation}
and define the numbers $\gamma_i^{\mu,\bm c}(I)$ by
\[
\gamma_i^{\mu,\bm c}(I)= \dim_\kk \left(\widetilde H_{i-1} (\Delta^{\mu,\bm c}(I))\right),
\]
where $\widetilde H_\bullet(\Delta)$ denote the reduced homology groups of a simplicial complex $\Delta$, with coefficients in $\kk$.

\begin{example}\label{ex:D-521-000}
 Let $I$ be as in \eqref{eq:I-411-520}, let $\mu=(5,2,1)$ and let $\bmc=(0,0,0)$. The simplicial complex $\Delta^{\mu,\bm c}(I)$ can be identified with the intersection of the interval $[\mu\setminus \bm c ,\mu \setminus (\bm c + \ee_{[s]}) ]=[(5,2,1),(4,1,0)]$ with $P(I)$ in the poset $P_n$ (with the reversed order). The faces of the complex and the corresponding partitions are colored in red and are pictured below.
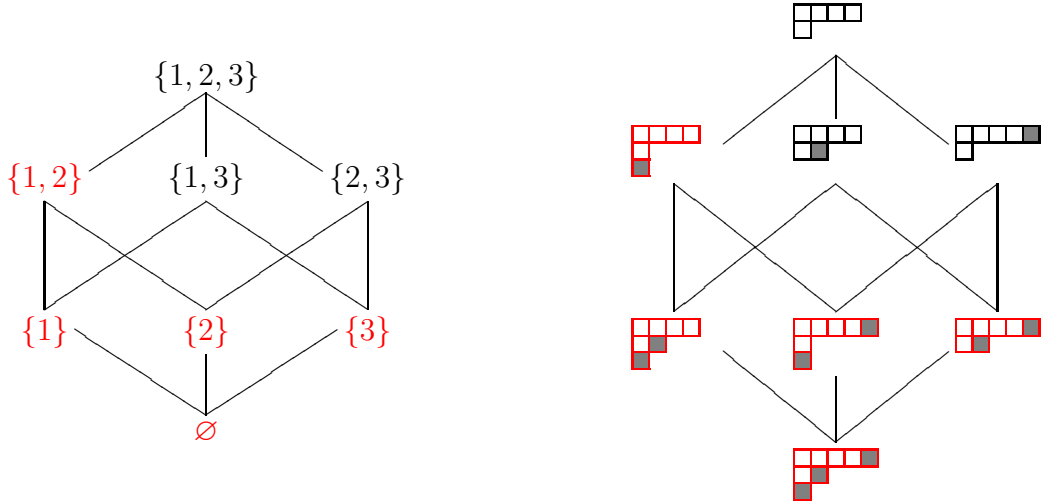
\[
\ytableausetup{boxsize=0.5em}
\begin{array}{c}
\xymatrix{
& *!D{\{1,2,3\}} \ar@{-}[dr] \ar@{-}[d] \ar@{-}[dl]& \\
*+<5pt>!D{\color{red}\{1,2\}} \ar@{-}[dr]\ar@{-}[d] & *+<5pt>!D{\{1,3\}} \ar@{-}[dr]\ar@{-}[dl] & *+<5pt>!D{\{2,3\}} \ar@{-}[dl]\ar@{-}[d]\\ 
*+<5pt>!U{\color{red}\{1\}} & *+<5pt>!U{\color{red}\{2\}} & *+<5pt>!U{\color{red}\{3\}} \\ 
& *+<5pt>!U{\color{red}\varnothing} \ar@{-}[ul] \ar@{-}[u] \ar@{-}[ur]& \\
}
\end{array}
\qquad\qquad\qquad
\begin{array}{c}
\xymatrix{
& *!D{\ydiagram[*(gray)] {4+0,1+0,0+0}*{4,1}} \ar@{-}[dr] \ar@{-}[d] \ar@{-}[dl]& \\
*+<5pt>!D{\color{red}\ydiagram[*(gray)] {4+0,1+0,0+1}*{4,1}} \ar@{-}[dr]\ar@{-}[d] & *+<5pt>!D{\ydiagram[*(gray)] {4+0,1+1,0+0}*{4,1}} \ar@{-}[dr]\ar@{-}[dl] & *+<5pt>!D{\ydiagram[*(gray)] {4+1,1+0,0+0}*{4,1}} \ar@{-}[dl]\ar@{-}[d]\\ 
*+<5pt>!U{\color{red}\ydiagram[*(gray)] {4+0,1+1,0+1}*{4,1}} & *+<5pt>!U{\color{red}\ydiagram[*(gray)] {4+1,1+0,0+1}*{4,1}} & *+<5pt>!U{\color{red}\ydiagram[*(gray)] {4+1,1+1,0+0}*{4,1}} \\ 
& *+<5pt>!U{\color{red}\ydiagram[*(gray)] {4+1,1+1,0+1}*{4,1}} \ar@{-}[ul] \ar@{-}[u] \ar@{-}[ur]& \\
}
\end{array}
\]
It follows that
\[\Delta^{\mu,\bm c}(I)= \big\{\{1,2\},\{1\},\{2\},\{3\},\varnothing\big\},\]
whose only non-vanishing reduced homology group is $\widetilde H_{0} (\Delta^{\mu,\bm c}(I))$, of dimension $\gamma_1^{\mu,\bm c}(I)=1$.
\end{example}

To introduce the final ingredient of our main result, we let $S^\lambda$ denote the Specht module associated with a partition $\lambda$ (see Section~\ref{subsec:Specht} for some background), which is a module over $\sym_{|\lambda|}$. We say that $\lambda$ is a \defi{hook partition} if $\lambda_i\leq 1$ for $i>1$ (the terminology is suggestive of the shape of the Young diagram of $\lambda$). For a sequence 
$$\pi=((p_1,1^{q_1}),\dots,(p_r,1^{q_r}))$$ of hook partitions with $\sum_{k=1}^r (p_k+q_k)=n$,
we write
\begin{align}
\label{eq:mcS-pi}
\mathcal S^\pi
= 
\mathrm{Ind}^{\sym_n}_{\sym_{p_1+q_1} \times \cdots \times \sym_{p_r+q_r}} \left(
S^{(p_1,1^{q_1})}\boxtimes \cdots \boxtimes S^{(p_r,1^{q_r})}
\right),
\end{align}
where $\boxtimes$ denotes the (external) tensor product of representations, and $\mathrm{Ind}$ denotes the induced representation.
We are now ready to state our main theorem.

\begin{theorem}
\label{thm:main}
Let $\mu=(d_1^{p_1},\dots,d_s^{p_s},0^{p_{s+1}}) \in P_n$ with $d_1>\cdots >d_s>0$ and let
$I$ be an $\sym_n$-invariant monomial ideal of $R$.
We have an isomorphism of $\kk$-vector spaces
{\large
\begin{equation}\label{eq:main-isomorphism}
\Tor_i(I,\kk)_{\langle \mu \rangle}
\cong \bigoplus_{\bm 0 \leq \bm c=(c_1,\dots,c_s) \leq \bm p(\mu)}
\left(\mathcal S^{( (p_1-c_1,1^{c_1}),\dots,(p_s-c_s,1^{c_s}),(p_{s+1}))}\right)^{\gamma_{i-|\bm c|}^{\mu,\bm c}(I)}.
\end{equation}
}Moreover, if $\mathrm{char}(\kk)=0$ or $\mathrm{char}(\kk)>n$, then \eqref{eq:main-isomorphism} is an isomorphism of $\sym_n$-modules.
\end{theorem}

In fact, our proof will show that (in arbitrary characteristic) $\Tor_i(I,\kk)_{\langle \mu \rangle}$ has a filtration by $\sym_n$-submodules, whose associated graded is isomorphic to the right side of \eqref{eq:main-isomorphism}. Since representations of $\sym_n$ are semisimple if $\mathrm{char}(\kk)=0$ or $\mathrm{char}(\kk)>n$, it follows that \eqref{eq:main-isomorphism} is an isomorphism of $\sym_n$-modules in these cases (we note that a similar issue arises in the calculation of $\Ext$ modules in \cite[Main Theorem]{Ra}). Theorem~\ref{thm:main} clarifies our earlier assertion that the building blocks of $\Tor_i(I,\kk)$ are induced representations of tensor products of Specht modules, and that their multiplicities are topological Betti numbers. We now illustrate Theorem~\ref{thm:main} with an example.

\begin{example}
\label{ex:1.2}
If $I$ is as in \eqref{eq:I-411-520} then using Macaulay2 one can see that the Betti table of $I$ is{\small
$$
\begin{matrix}
&0&1&2\\
\text{total:}&9&12&4\\
\text{6:}&3&\text{.}&\text{.}\\
\text{7:}&6&6&\text{.}\\
\text{8:}&\text{.}&3&\text{.}\\
\text{9:}&\text{.}&3&3\\
\text{10:}&\text{.}&\text{.}&1
\end{matrix}
$$
}
and that
\begin{align*}
\Tor_0(I,\kk) &= \Tor_0(I,\kk)_{\langle (4,1,1)\rangle} \oplus \Tor_0(I,\kk)_{\langle (5,2,0)\rangle},\\
\Tor_1(I,\kk) &= \Tor_1(I,\kk)_{\langle (4,4,1)\rangle} \oplus \Tor_1(I,\kk)_{\langle (5,2,1)\rangle} \oplus \Tor_1(I,\kk)_{\langle (5,5,0)\rangle},
\\
\Tor_2(I,\kk) &= \Tor_2(I,\kk)_{\langle (4,4,4)\rangle} \oplus \Tor_2(I,\kk)_{\langle (5,5,1)\rangle}.
\end{align*}
To identify the $\sym_3$-module structure, one first computes the relevant complexes $\Delta^{\mu,\bmc}(I)$:
{\small
\begin{align*}
&\Delta^{(4,1,1),(0,0)}(I)=
\Delta^{(5,2,0),(0,0)}(I)= \{\varnothing\},\\
&
\Delta^{(4,4,1),(1,0)}(I)=
\Delta^{(5,5,0),(1)}(I)=
\{\varnothing\},\ \
\Delta^{(5,2,1),(0,0,0)}= \big\{\{1,2\},\{1\},\{2\},\{3\},\varnothing\big\},\\
&\Delta^{(4,4,4),(2)}(I)= \{\varnothing\} \mbox{ and }
\Delta^{(5,5,1),(1,0)}= \big\{\{1\},\{2\},\varnothing\big\}.
\end{align*}
}
We leave the details of this calculation to the reader, noting that the description of $\Delta^{(5,2,1),(0,0,0)}(I)$ was explained in Example~\ref{ex:D-521-000}. We then have
\begin{align*}
& \gamma_0^{(4,1,1),(0,0)}(I)=
\gamma_0^{(5,2,0),(0,0)}(I)=1,\\
& \gamma_0^{(4,4,1),(1,0)}(I)=
\gamma_1^{(5,2,1),(0,0,0)}(I)=
\gamma_0^{(5,5,0),(1)}(I)=1,\mbox{ and }\\
& \gamma_0^{(4,4,4),(2)}(I)=
\gamma_1^{(5,5,1),(1,0)}(I)=1.
\end{align*}
Theorem \ref{thm:main} implies based on these computations that we have
\ytableausetup{boxsize=0.3em}
\begin{align*}
&
\Tor_0(I,\kk)_{\langle (4,1,1)\rangle} \cong \mathrm{Ind}_{\sym_1\times \sym_2}^{\sym_3} \left (S^{\ydiagram 1}\boxtimes S^{\ydiagram 2} \right),\\
&\Tor_0(I,\kk)_{\langle (5,2,0)\rangle} \cong \mathrm{Ind}_{\sym_1 \times \sym_1 \times \sym_1}^{\sym_3} \left (S^{\ydiagram 1}\boxtimes S^{\ydiagram 1}\boxtimes S^{\ydiagram 1} \right),\\
&\Tor_1(I,\kk)_{\langle (4,4,1)\rangle} \cong \mathrm{Ind}_{\sym_2\times \sym_1}^{\sym_3} \left (S^{\ydiagram {1,1}}\boxtimes S^{\ydiagram 1} \right),\\
&\Tor_1(I,\kk)_{\langle (5,2,1)\rangle} \cong \mathrm{Ind}_{\sym_1\times \sym_1\times \sym_1}^{\sym_3} \left (S^{\ydiagram 1}\boxtimes S^{\ydiagram 1}\boxtimes S^{\ydiagram 1} \right),\\
&\Tor_1(I,\kk)_{\langle (5,5,0)\rangle} \cong \mathrm{Ind}_{\sym_2\times \sym_1}^{\sym_3} \left (S^{\ydiagram {1,1}}\boxtimes S^{\ydiagram 1}\right),\\
&\Tor_2(I,\kk)_{\langle (4,4,4)\rangle} \cong S^{\ydiagram {1,1,1}},\\
&\Tor_2(I,\kk)_{\langle (5,5,1)\rangle} \cong \mathrm{Ind}_{\sym_2\times \sym_1}^{\sym_3} \left (S^{\ydiagram {1,1}}\boxtimes S^{\ydiagram 1}\right).
\end{align*}
\end{example}

As a first application of Theorem~\ref{thm:main}, we explain how to determine the $\sym_n$-invariant part of $\Tor_i(I,\kk)$. We express $I$ as in \eqref{eq:explicit-sym-invariant-I}, and define the \defi{unsymmetrization} of $I$ to be the ideal
\[ J = (x^{\lambda^1},\cdots,x^{\lambda^r}) \subseteq R.\]
The relationship between the $\Tor$ groups of $I$ and $J$ is given by the following.

\begin{theorem}\label{thm:invariant-intro}
If $\mathrm{char}(\kk)=0$ or $\mathrm{char}(\kk)>n$, then for partitions $\lambda^1,\cdots,\lambda^r\in P_n$ we have
\[
    \Tor_i \big(\langle \lambda^1,\cdots, \lambda^r \rangle_{\sym_n},\kk \big)^{\sym_n} \cong \Tor_i\big((x^{\lambda^1},\cdots,x^{\lambda^r}),\kk\big) \ \mbox{ for all }i.
\]
\end{theorem}

The proof of Theorem~\ref{thm:invariant-intro} is explained in Section~\ref{sec:sn-inv-Betti} as an application of Theorem~\ref{thm:main} and the Nerve Theorem. Here, we illustrate Theorem~\ref{thm:invariant-intro} with an example.

\begin{example}
If $I$ is as in \eqref{eq:I-411-520}, then its unsymmetrization is the ideal $J=(x_1^4x_2x_3,x_1^5x_2^2)$. Since the generators of $J$ have a unique syzygy (coming from their lcm $x_1^5x_2^2x_3$), we get
$$\Tor_0(J,\kk)=\Tor_0(J,\kk)_{(4,1,1)}\oplus \Tor_0(J,\kk)_{(5,2,0)} \cong \kk \oplus \kk$$
and
$$\Tor_1(J,\kk)=\Tor_1(J,\kk)_{(5,2,1)} \cong \kk. $$
Theorem \ref{thm:invariantpart} implies that
$$\Tor_0(I,\kk)^{\sym_3}=\Tor_0(I,\kk)_{\langle(4,1,1)\rangle}^{\sym_3} \oplus \Tor_0(I,\kk)_{\langle (5,2,0)\rangle }^{\sym_3} \cong \kk \oplus \kk$$
and
$$\Tor_1(I,\kk)^{\sym_3}=\Tor_1(I,\kk)^{\sym_3}_{\langle (5,2,1)\rangle} \cong \kk,$$
which can be checked (based on Pieri's rule) using the computations in Example \ref{ex:1.2}: the trivial $\sym_3$-module $S^{\ydiagram {3}}$ appears as a summand in $\mc{S}^{\pi}$ if and only if each of the partitions in $\pi$ has a single row (in which case the multiplicity of $S^{\ydiagram {3}}$ is one).
\end{example}

\begin{remark}
Theorem \ref{thm:invariantpart} implies that if $\langle \lambda^1,\dots,\lambda^r\rangle_{\sym_n}$ has a linear resolution, then so does $(x^{\lambda^1},\dots,x^{\lambda^r})$.
A combinatorial characterization of $\sym_n$-invariant monomial ideals having a linear resolution was given in \cite{Ra}. These are exactly the symmetric shifted ideals (generated in a single degree) defined in \cite{BDAG}.
It follows that the unsymmetrizations of these ideals have a linear resolution.
\end{remark}

Theorem~\ref{thm:main} gives a concrete recipe for computing the multigraded Betti numbers
\[\beta_{i,\bma}(I) = \dim\Tor_i(I,\kk)_{\bma},\]
but in practice, the difficulty of the calculation depends on the complexity of evaluating the homology of~$\Delta^{\mu,\bmc}$. As shown by example in \cite[Section~5]{Mu}, the numbers $\beta_{i,\bma}(I)$ may depend on the characteristic of $\kk$. However, the shape of the Betti table, as measured by the Castelnuovo--Mumford regularity $\reg(I)$, and by the projective dimension $\pdim(I)$, does not depend on $\mathrm{char}(\kk)$! This was first shown in \cite{Ra}, and in Section~\ref{sec:prim-dec-extr-Betti} we give an equivalent (but somewhat simpler) recipe for computing $\reg(I)$ and $\pdim(I)$. We also extend the notion of \defi{extremal Betti numbers} from \cite{BCP} to our context, and compute the extremal Betti numbers of $I$ in Theorem~\ref{thm:extremal}.

The results of our work relate to the broader context of the study of finiteness properties of ideals in an infinite polynomial ring, which are invariant under a large group of symmetries. A significant body of research has been performed in recent years on finite generation statements, most often under the designation \emph{Noetherianity up to symmetry} or \emph{representation stability}, and has had important applications including two (of several) recent proofs of Stillman's conjecture on the projective dimension (and regularity) of polynomial ideals \cite{DLL,ESS}. In the case of the infinite polynomial ring $R_{\infty} = \kk[x_1,x_2,\cdots]$, with the action of the infinite symmetric group $\sym_{\infty}$ by coordinate permutations, ideals $I_{\infty}\subseteq R_{\infty}$ that are $\sym_{\infty}$-invariant are generated by finitely many $\sym_{\infty}$-orbits (this is a classical result due to Cohen \cite{Cohen}, rediscovered more recently in \cite{AH-finite,HS-finite}). It is then natural to explore finiteness beyond the set of generators, and to understand how it is reflected in other homological invariants. To that end, we let $f_1,\dots,f_r \in \kk[x_1,\dots,x_n]\subset R_{\infty}$ be polynomials whose $\sym_{\infty}$-orbits generate $I_{\infty}$,
and consider the sequence of ideals
\begin{equation}\label{eq:def-Im}
I_m=(\sigma(f_i) \mid 1 \leq i \leq r,\ \sigma \in \sym_m) \subseteq \kk[x_1,\dots,x_m]\text{ for }m\geq n.
\end{equation}
One can guess that the finiteness properties of $I_{\infty}$ are reflected by uniform behaviors of homological invariants of the ideals $I_m$. For instance, it is shown in \cite[Corollary~3.12]{DNNR-codim} that the (co)dimension of the ideals $I_m$ is computed by a linear function when $m\gg 0$, and it is conjectured in \cite[Conjecture~1.3]{DNNR-codim} that the same result is true for the projective dimension $\pdim(I_m)$. Similarly, it is conjectured in \cite[Conjecture~1.1]{DNNR-reg} that $\reg(I_m)$ is a linear function when $m\gg 0$. In the case when $f_1,\cdots,f_r$ are monomials, the linearity of $\pdim(I_m)$ and $\reg(I_m)$ is established by the authors in \cite[Corollary~1.2]{Mu} and \cite[Theorem~6.1]{Ra}. In Section~\ref{sec:vary-vars}, we extend these results to each of the Betti numbers of the ideals $I_m$, by providing a concrete recipe of how these numbers change as we vary $m$. Exhibiting a uniform behavior for the Betti numbers of $I_m$, when $f_1,\cdots,f_r$ are no longer assumed to be monomials, remains a significant open problem, and we hope that our work will inspire further investigations in this direction.

\bigskip

\noindent{\bf Organization.}
In Section~\ref{sec:prelim} we introduce the necessary notation and preliminary results regarding partitions, simplicial complexes, Betti numbers, Specht modules, and multidimensional chain complexes. In Section~\ref{sec:main-thm} we prove Theorem~\ref{thm:main}, and in Section~\ref{sec:sn-inv-Betti} we explain the proof of Theorem~\ref{thm:invariant-intro}.
In Section~\ref{sec:prim-dec-extr-Betti} we discuss the primary decomposition, extremal Betti numbers, regularity and projective dimension of $\sym_n$-invariant monomial ideals. Finally, in Section~\ref{sec:vary-vars} we explain how the Betti numbers change as we increase the number of variables.

\section{Preliminaries}
\label{sec:prelim}
In this section, we introduce some basic notation which will be used in the paper.

\subsection{Some remarks on partitions and multidegrees}

Let $I$ be an $\sym_n$-invariant monomial ideal of $R$.
Recall that $P(I)=\{\lambda \in P_n \mid x^\lambda \in I\}$.
Since $x^{\bm a} \in I$ if and only if ${\mathrm{part}(\bm a)}\in P(I)$,
the set $P(I)$ determines the ideal $I$.
We regard $P_n$ as a poset with the relation defined by $(a_1,\dots,a_n) \geq (b_1,\dots,b_n)$ if $a_i \geq b_i$ for all $i=1,2,\dots,n$.
Let $\Lambda (I)$ be the set of minimal elements in $P(I)$. Identifying partitions with the corresponding monomials in $R$, we have that up to the action of $\sym_n$, the set $\Lambda (I)$ forms a minimal set of generators of $I$: we have $I=\langle \lambda \mid \lambda \in \Lambda(I)\rangle_{\sym_n}$ and no proper subset of $\Lambda(I)$ generates $I$. In particular, $\lambda \in P_n$ is contained in $P(I)$ if and only if there is $\mu \in \Lambda(I)$ such that $\lambda \geq \mu$.

We sometimes regard a partition $\lambda$ as an element of $\ZZ^n$.
To avoid confusion,
for a partition $\lambda=(\lambda_1,\dots,\lambda_n)$, 
when we denote the graded component of a module $M$ of degree $(\lambda_1,\dots,\lambda_n)\in \ZZ^n$,
we write it as $M_{\bm \lambda}$ instead of $M_\lambda$.
Also, when we write partitions, we sometimes ignore ``$0$" and identify $(\lambda_1,\dots,\lambda_n)$ and $(\lambda_1,\dots,\lambda_n,0,\dots,0)$. For any $\bma=(a_1,\cdots,a_n)\in\bb{Z}^n$ we write $|\bma|=a_1+a_2+\cdots+a_n$ for the \defi{size} of $\bma$.

We often consider both $\ZZ^n$ and $\ZZ^s$.
We write $\ee_1,\dots,\ee_s$ for the standard vectors of $\ZZ^s$ and write $\eee_1,\dots,\eee_n$ for the standard vectors of $\ZZ^n$.
Also, for subsets $F \subset [s]$ and $G \subset [n]$, we write $\ee_F=\sum_{i \in F} \ee_i$ and $\eee_G=\sum_{i \in G} \eee_i$.

\subsection{Simplicial complexes and their homology groups}
\label{subsec:simpl-cpx}
A simplicial complex on a finite set $V$ is a collection $\Delta$ of subsets of $V$ satisfying the condition that $F \in \Delta$ and $G\subset F$ imply $G \in \Delta$.
Elements of $\Delta$ are called \defi{faces} and maximal elements of $\Delta$ are called \defi{facets}.
We distinguish the empty simplicial complex $\varnothing$ from the simplicial complex $\{\varnothing\}$ consisting only of the empty face.

If $\Delta$ is a simplicial complex on $V=\{v_1,\dots,v_n\}$, we write
$$\widetilde C_\bullet : 0 
\longleftarrow C_{-1}(\Delta)
\stackrel{\partial} \longleftarrow
C_{0}(\Delta)
\stackrel{\partial} \longleftarrow
C_{1}(\Delta)
\stackrel{\partial} \longleftarrow \cdots$$
for the (reduced) simplicial chain complex of $\Delta$ over a field $\kk$.
Here, each $C_k(\Delta)$ is the $\kk$-vector space spanned by the symbols $\{\alpha_F \mid F \in \Delta,\ |F|=k+1\}$. If we consider the total order $v_1< \cdots <v_n$ on $V$ then the boundary map is given by 
$$\partial(\alpha_F)=\sum_{v \in F} \epsilon_v(F) \cdot \alpha_{F \setminus \{v\}}$$
where $\epsilon_v(F)=(-1)^{|\{u \in F \mid u \leq v\}|}$.
The homology $\widetilde H_i(\Delta)=H_i(\widetilde C_\bullet(\Delta))$ is called the $i$-th \defi{reduced homology group} of $\Delta$.
We note that $\widetilde H_{-1}(\{\varnothing\}) \cong \kk$ while $\widetilde H_{i}(\varnothing)=0$ for all $i$.

\subsection{Betti numbers of monomial ideals}
It is known that when $I$ is a monomial ideal,
the $\ZZ^n$-graded components of $\Tor_i(I,\kk)$ can be identified with reduced homology groups of certain simplicial complexes.
We quickly recall this fact.

Let $I \subset R$ be a monomial ideal.
Let $K_\bullet^R=K_\bullet^R(x_1,\dots,x_n)$ be the Koszul complex w.r.t.\ the variables $x_1,\dots,x_n$ and $K_\bullet(I)=K_\bullet^R(x_1,\dots,x_n) \otimes_R I$.
We have that $K^R_i$ is the free $R$-module whose basis is the set $\{e_{a_1} \wedge \cdots \wedge e_{a_i} \mid 1 \leq a_1< \cdots <a_i \leq n\}$, where $e_{a_1} \wedge \cdots \wedge e_{a_i}$ is an element of the exterior algebra generated by $e_1,\dots,e_n$.
For $\bm a=(a_1,\dots,a_n) \in \ZZ_{\geq 0}^n$, we define the simplicial complex
$$\Delta_{\bm a}^I=\{ F \subset [n] \mid x^{\bm a-\eee_F} \in I\}.$$
One has an identification
\[
K_\bullet(I)_{\bm a} \cong \widetilde C_{\bullet+1}(\Delta_{\bm a}^I)
\]
given by the correspondence $x^{\bm a- \eee_{\{i_1,\dots,i_k\}}} e_{i_1} \wedge \cdots \wedge e_{i_k} \to \alpha_{\{i_1,\dots,i_k\}}$,
for $i_1<\dots<i_k$.
Since $\Tor_i(I,\kk) \cong H_i(K_\bullet(I))$,
we obtain the following formula.
 
\begin{theorem}[{\cite[Proposition 1.1]{BH2}}] For any monomial ideal $I \subset S$ and $\bm a \in \ZZ_{\geq 0}^n$, 
$$\Tor_i(I,\kk)_{\bm a} \cong \widetilde H_{i-1}(\Delta^I_{\bm a}).$$
\end{theorem}

For square-free monomial ideals, the above formula coincides via Alexander duality with Hochster's formula (see \cite[\S 1]{BH2}).

\subsection{Specht modules of hook partitions}
\label{subsec:Specht}

Here we explain a few basic facts on Specht modules.
We only consider those modules for hook partitions since these are the only cases which we need.
We refer the readers to \cite{James,Sa} for a general theory.

A \defi{hook partition} is a partition of the form $(p,1^q)$ with $p\geq 1$ and $q \geq 0$.
Let $\lambda=(1+t,1^{s-1})$ be a hook partition with $s \geq 1,t \geq 0$.
A (Young) \defi{tableau} of shape $\lambda$ is an assignment of distinct positive integers to each box in $\lambda$.
We consider the following relation $\sim$ (extended linearly) on the vector space spanned by tableaux of shape $\lambda$.
\ytableausetup{boxsize=1.4em}
\begin{itemize}
\item[(I)]
For any permutation $\sigma$ on $[s]$ with $\sigma(k)=p_k$ and for any permutation $\tau$ on $[t]$ with $\tau(k)=q_k$, we have
$$
\left( \begin{ytableau} a_1 & b_1 &\none[\cdots] & b_t\\
a_2\\
\none[\vdots]\\
a_s \end{ytableau} \right) \sim \mathrm{sgn}(\sigma) \cdot 
\left(\begin{ytableau} a_{p_1} & b_{q_1} &\none[\cdots] & b_{q_t}\\
a_{p_2}\\
\none[\vdots]\\
a_{p_s} \end{ytableau}\right).
$$
\item[(II)]
For any sequence of integers $a_0,a_1,\dots,a_s,b_2,\dots,b_t$, we have
$$
\sum_{i=0}^s (-1)^i \cdot
\left(\begin{ytableau} a_0 & a_i &b_2 &\none[\cdots] & b_t\\
\none[\vdots]\\
\scriptstyle a_{i-1}\\
\scriptstyle a_{i+1}\\
\none[\vdots]\\
a_s \end{ytableau}\right) \sim 0.
$$
\end{itemize}
Let $\mathrm{Tab}(\lambda)$ be the set of tableaux of shape $\lambda$ with entries $1,2,\dots,s+t$.
The quotient space $S^{\lambda}=(\mathrm{span}_\kk(\mathrm{Tab}(\lambda)))/\sim$ is called the \defi{Specht module} of shape $\lambda$.
We say that a tableau
\ytableausetup{boxsize=1.2em}
$$\begin{ytableau} a_1 & b_1 & \none[\cdots] & b_t\\ a_2 \\ \none[\vdots] \\ a_s \end{ytableau}$$
is \defi{standard} if $a_1 < \cdots <a_s$ and $a_1< b_1< \cdots <b_t$.
It is well-known that standard tableaux in $\mathrm{Tab}(\lambda)$ form a basis of $S^\lambda$.

\subsection{Multi-dimensional complexes}
\label{subsec:multi-cx}

We define a \defi{$\bb{Z}^s$-complex} of $\kk$-vector spaces to be a complex $(K_{\bullet},\pd)$ where each term has a decomposition
\begin{equation}\label{eq:decomp-Kl}
 K_l = \bigoplus_{c_1+\cdots+c_s=l} K_{(c_1,\cdots,c_s)}
\end{equation}
and the differential $\pd$ sends
\[ \pd(K_{(c_1,\cdots,c_s)}) \subseteq \bigoplus_{i=1}^s K_{(c_1,\cdots,c_i-1,\cdots,c_s)}.\]
All our complexes are finite, that is, $K_{\bmc}\neq 0$ for finitely many $\bmc\in\bb{Z}^s$, and the spaces $K_{\bmc}$ are finite dimensional. We have a decomposition $\pd = \bigoplus\pd^i_{\bmc}$ with $1\leq i\leq s$ and $\bmc\in\bb{Z}^s$, where
\[ \pd^i_{\bmc}:K_{\bmc} \lra K_{\bmc-\ee_i},\]
and the condition that $\pd$ is a differential translates into
\[
\pd^i_{\bmc-\ee_i} \circ \pd^i_{\bmc}=0,\text{ and }\pd^j_{\bmc-\ee_i}\circ \pd^i_{\bmc} + \pd^i_{\bmc-\ee_j}\circ\pd^j_{\bmc} = 0\text{ for }i\neq j.
\]
We write $\pd^K$ and $\pd^{K,i}_{\bmc}$ when we want to emphasize the complex that $\pd$ is a differential of. We define a morphism of $\bb{Z}^s$-complexes $f:K\lra K'$ to be a morphism of complexes which is compatible with (\ref{eq:decomp-Kl}). Equivalently, for each $\bmc\in\bb{Z}^s$ we have a $\kk$-linear map $f_{\bmc}:K_{\bmc}\lra K'_{\bmc}$, satisfying
\[ \pd^{K',i}_{\bmc}\circ f_{\bmc} = f_{\bmc-\ee_i}\circ\pd^{K,i}_{\bmc}\text{ for all }1\leq i\leq s,\ \bmc\in\bb{Z}^s.\]
We write $\com_s$ for the category of $\bb{Z}^s$-complexes, identify $\com_0$ with the category $\vecs$ of $\kk$-vector spaces, and note that $\com_1$ is the usual category of complexes associated to $\vecs$. If we write $\mf{E}=\bw^{\bullet}(\pd^1,\cdots,\pd^s)$ for the exterior algebra on $\pd^1,\cdots,\pd^s$, then $\mf{E}$ has a $\bb{Z}^s$-grading with $\deg(\pd^i)=-\ee_i$, and the notion of a $\bb{Z}^s$-complex is equivalent to that of a finitely generated $\bb{Z}^s$-graded $\mf{E}$-module. We will write $\grmod_{\mf{E}}$ for the category of such modules, and use freely the equivalence between $\com_s$ and $\grmod_{\mf{E}}$.


We define the \defi{support} of a $\bb{Z}^s$-complex to be the set
\begin{equation}\label{eq:def-supp}
\supp(K) = \{ \bmc\in\bb{Z}^s: K_{\bmc}\neq 0\}.
\end{equation}
For $\bmd\in\bb{Z}^s$, we define the shifted complex $K[\bmd]$ by $K_{\bmc} = K_{\bmc+\bmd}$, with differentials shifted accordingly. We can think of a vector space $W$ as a $\bb{Z}^s$-complex supported at $(0^s)$, and we write $W[\bmd]$ for the corresponding shift (which is supported at $-\bmd$). 

\begin{example}\label{ex:boolean}
 For a vector space $W$, we define the $\bb{Z}^s$-complex $E=E(W)$ by
 \[ E_{\bmc} = \begin{cases}
 W & \text{if }\bmc\in\{0,1\}^s; \\
 0 & \text{otherwise}.
 \end{cases}
 \]
 For $\bmc\in\supp(E)$ with $c_i=1$, we define $\pd^i_{\bmc} : W \lra W$ to be multiplication by $(-1)^{c_1+\cdots+c_i}$. It is easy to see that $E$ is an exact complex (isomorphic up to shift to the tensor product of $W$ with the reduced chain complex of a simplex). As an object of $\grmod_{\mf{E}}$, $E$ can be identified with the free module $W\oo_{\kk}\mf{E}(-1^s)$, with generators $W$ in degree $(1^s)$. As such, $E$ is a projective object of $\com_s$ (it is also injective by \cite[Proposition~7.19]{eis-geom-syz}). 
\end{example}

We define a \defi{Boolean $\bb{Z}^s$-complex} to be a $\bb{Z}^s$-complex $K$ which is isomorphic to $E[\bmd]$, where $\bmd\in\bb{Z}^s$ and $E$ is as in Example~\ref{ex:boolean}. This is equivalent to the fact that $\supp(K) = -\bmd+\{0,1\}^{\times n}$ and $\pd^i_{\bmc}:K_{\bmc} \lra K_{\bmc-\ee_i}$ is an isomorphism whenever $\bmc,\bmc-\ee_i \in \supp(K)$. Whenever we want to emphasize the relation between the Boolean complex $K$ and the reduced chain complex of a simplex, we will write for each subset $F\subseteq[s]$
\begin{equation}\label{eq:def-K-F}
    K_{F} = K_{-\bmd+\ee_F}.
\end{equation}

It follows from Example~\ref{ex:boolean} that Boolean complexes are projective (and injective) objects in $\com_s$, which has the following useful consequence.

\begin{corollary}\label{cor:filtrations}
 Suppose that $K$ is a $\bb{Z}^s$-complex with a filtration whose composition factors $E^1,\cdots,E^r$ are Boolean $\bb{Z}^s$-complexes. We have that
 \[ K \simeq E^1 \oplus \cdots \oplus E^r.\]
\end{corollary}

If $K\in\com_s$ and $L\in\com_t$, then the (external) tensor product $K\bt L$ is defined to be the complex in $\com_{s+t}$ with
\[(K\bt L)_{\bmc} = \bigoplus_{\bmd+\bme=\bmc}K_{\bmd}\oo_{\kk} L_{\bme},\]
and differential
\[\pd^{K\bt L}_{\bmc}(u\oo v) = \pd^K(u) \oo v + (-1)^{|\bmd|} u\oo\pd^K(v)\text{ for }u\in K_{\bmd},\ v\in L_{\bme}.\]
We note that the complex in Example~\ref{ex:boolean} is the tensor product of the $\bb{Z}$-complex $W\simeq W$ with $(s-1)$ copies of the $\bb{Z}$-complex $\kk\simeq\kk$. In general, the tensor product of a Boolean $s$-complex with a Boolean $t$-complex is a Boolean $(s+t)$-complex.

If $\mc{F}^{\bullet}(K)$ is an decreasing filtration of $K$ by $\bb{Z}^s$-subcomplexes, we write 
\[\gr^k(K) = \mc{F}^k(K)/\mc{F}^{k+1}(K)\text{ and }\gr(K) = \bigoplus_{k} \gr^k(K).\]
Given filtrations $\mc{F}^{\bullet}(K)$, $\mc{F}^{\bullet}(L)$, we get an induced filtration on $K\bt L$, with
\[ \mc{F}^k(K\bt L) = \sum_{i+j=k} \mc{F}^i(K) \bt \mc{F}^j(L)\text{ and }\gr^k(K\bt L) = \bigoplus_{i+j=k} \gr^i(K) \bt \gr^j(L).\]

If $G$ is a group, we will be interested more generally in the category $\com_s^G$ of finite complexes of finite $\kk[G]$-modules, or equivalently, the category $\grmod_{\mf{E}}^G$ of finitely generated $G$-equivariant $\bb{Z}^s$-graded $\mf{E}$-modules, where the action of $G$ on $\mf{E}$ is trivial. For $s=0$, $\com_0^G$ is the category $\mods_G$ of finite $G$-modules. If $K\in\com_s^G$ and $L\in\com_t^{G'}$ then $K\bt L\in\com_{s+t}^{G\times G'}$, and the discussion of filtrations is analogous in the equivariant setting.

Using the natural isomorphisms
\[\Hom_{\com_s^G}(E(W),K) \simeq \Hom_{\grmod_{\mf{E}}^G}(W\oo_{\kk}\mf{E}(-1^s),K) \simeq \Hom_{\mods_G}(W,K_{(1^s)})\]
we can interpet the construction of the Boolean complex $E(W)$ in Example~\ref{ex:boolean} as a functor $E:\mods_G\lra \com_s^G$ which is left-adjoint to the functor $P:\com_s^G\lra\mods_G$ given by $P(K) = K_{(1^s)}$. Since $P$ is exact, we have that $E(W)$ is projective whenever $W$ is a projective $G$-module. When $G$ is a finite group and $\kk$ has characteristic zero or coprime to $|G|$, we have that $\mods_G$ is semi-simple, and in particular $E(W)$ is a projective object of $\com_s^G$ for every $G$-module $W$. We get the following equivariant version of Corollary~\ref{cor:filtrations}:

\begin{corollary}\label{cor:G-filtrations}
 Suppose that $G$ is a finite group, $\kk$ is a field of characteristic zero or coprime to $|G|$, and $K\in\com_s^G$ has a filtration with composition factors $E^i = E(W_i)$, where $W_i\in\mods_G$, for $i=1,\dots,r$. We have that
 \[ K \simeq E^1 \oplus \cdots \oplus E^r.\]
\end{corollary}

We end this section by explaining how Corollaries~\ref{cor:filtrations} and ~\ref{cor:G-filtrations} will be applied in our work. Suppose that $F\in\com_1^G$ is an exact complex supported in non-negative degrees:
\[ 0\lla F_0 \overset{\pd_1}{\lla} F_{1} \lla \cdots \overset{\pd_r}{\lla} F_r \lla 0.\]
We let $U_l = \operatorname{Im}(\pd_{l+1})$ and $D_l=F_l/U_l$, so that $\pd_l$ establishes an isomorphism $D_l\simeq U_{l-1}$. We have a natural filtration on $F$ given by the canonical truncations
\[ \mc{F}^l(F):\quad 0 \lla U_l \lla F_{l+1} \lla \cdots \lla F_{r-1} \lla F_r \lla 0,\]
with $\gr^l(F) = (D_{l+1}\simeq U_l)$ a Boolean $\bb{Z}$-complex. By Corollary~\ref{cor:filtrations}, we have an isomorphism $F\simeq\gr(F)$ in $\com_1$, which by Corollary~\ref{cor:G-filtrations} can be taken to be $G$-equivariant (that is, in $\com_1^G$) if $G$ is finite and $\kk$ has characteristic zero or coprime to $|G|$. Choosing (not necessarily $G$-equivariant) sections of the quotient maps $F_l\twoheadrightarrow D_l$, we can picture the complex $F$ as:
\[
\xymatrix @C=4pc @R=0pc{
U_0 & U_1 & U_2 & U_3 & \\
& \bigoplus & \bigoplus & \bigoplus & \hspace{-50pt}\cdots \\
 & D_1 \ar[uul]_-{\simeq}^-{\pd_1} & D_2 \ar[uul]_-{\simeq}^-{\pd_2} & D_3 \ar[uul]_-{\simeq}^-{\pd_3} & \\
}
\]

More generally, if $F^i\in\com_1^{G_i}$ for $i=1,\dots,s$, then the canonical filtrations on each $F^i$ induce a filtration on the tensor product $F=F^1\bt\cdots\bt F^s\in\com_s^G$, where $G=G_1\times\cdots\times G_s$. We have an isomorphism $F\simeq\gr(F)$ in $\com_s$, and if $G$ is finite and $\kk$ has characteristic zero or coprime to $|G|$, then $F\simeq\gr(F)$ in $\com_s^G$.

\section{Proof of the main theorem}
\label{sec:main-thm}

The goal of this section is to prove Theorem~\ref{thm:main}. We first study the complex
\begin{equation}\label{eq:def-K-mu}
 K_\bullet^{\mu} = \bigoplus_{\bm a \in \ZZ^n,\ \mathrm{part}(\bm a)=\mu} (K_\bullet^R)_{\bm a},
\end{equation}
where $\mu\in P_n$. We note that $K_\bullet^{\mu}$ is a complex of $\sym_n$-modules, and using the notation in Section~\ref{subsec:multi-cx}, we will show that $K^{\mu}_{\bullet}$ can be thought of as an object in $\com_s^{\sym_n}$ for an appropriate value of $s$, and that $K^{\mu}_{\bullet}$ has a natural filtration with composition factors that are $\sym_n$-equivariant Boolean complexes. Based on the discussion in Section~\ref{subsec:multi-cx}, this gives a decomposition of $K^{\mu}_{\bullet}$ into a direct sum of Boolean complexes, which is $\sym_n$-equivariant in characteristic zero or $>n$. This decomposition is then the key ingredient in the proof of Theorem~\ref{thm:main}.

\medskip

\noindent{\bf Step 1.}
Suppose that $\mu=(a^n)$ with $a >0$. We have
$$
K^{\mu}_l = \mathrm{span}_\kk \left\{ \sigma\big((x_1^{a-1}\cdots x_l^{a-1}x_{l+1}^a\cdots x_n^a) \cdot (e_1\wedge \cdots \wedge e_l)\big) \mid \sigma \in \sym_n\right\}.$$
Since $K^R_{\bullet}$ is exact except in degree $(0^n)$, it follows that $K^{\mu}_{\bullet}$ is also exact. We can then define $U_l=\operatorname{Im}(\pd_{l+1})$ and $D_l=K^{\mu}_l/U_l$ as in Section~\ref{subsec:multi-cx}, to get a filtration of $K^{\mu}_{\bullet}$ by Boolean $\sym_n$-equivariant complexes. To determine the isomorphism type of each $D_l$ as an $\sym_n$-module, we note that it does not depend on $a$, and hence we can take $a=1$. The natural map that associates
\[\ytableausetup{boxsize=2em}
\begin{ytableau} a_1 & b_1 &\none[\cdots] & b_{n-l}\\
a_2\\
\none[\vdots]\\
a_l \end{ytableau} \lra\quad x_{b_1}\cdots x_{b_{n-l}}\cdot e_{a_1}\wedge\cdots\wedge e_{a_l}
\]
induces an isomorphism between the Specht module $S^{(n-l+1,1^{l-1})}$ and $D_l$ (the relations (I) correspond to the skew-symmetric property of wedge products, while the relations (II) correspond to the generators $\pd_{l+1}(x_{b_2}\cdots x_{b_{n-l}} e_{a_0}\wedge\cdots\wedge e_{a_l})$ of $U_l$). Using that $K^{\mu}_l \simeq \mathrm{Ind}_{\sym_l \times \sym_{n-l}}^{\sym_n} \left( S^{(1^l)} \boxtimes S^{(n-l)} \right)$ (see \cite[Section~4]{Ga}) we recover a special case of the filtrations in \cite[\textsection~16]{James} known as Pieri's rule: there is an exact sequence
\[ 0 \lra S^{(n-l,1^l)} \lra \mathrm{Ind}_{\sym_l \times \sym_{n-l}}^{\sym_n} \left( S^{(1^l)} \boxtimes S^{(n-l)} \right) \lra S^{(n-l+1,1^{l-1})} \lra 0,\]
given by the inclusion of $U_l\simeq D_{l+1}$ into $K^{\mu}_l$, followed by the projection onto $D_l$.

\medskip

\noindent{\bf Step 2.}
We now consider the general case, when $\mu$ is of the form
$$\mu=(\mu_1,\dots,\mu_n)=(d_1^{p_1},\dots,d_s^{p_s},0^{p_{s+1}})\in P_n,$$
where $d_1> \cdots>d_s>0$. We set $d_{s+1}=0$ and let
\begin{equation}\label{eq:def-Xk}
X_k=\{ x_i \mid \mu_i=d_k\}\ \  \mbox{ and }
\ \ \kk[X_k]=\kk[x_i : x_i \in X_k],\text{ for }k=1,\cdots,s+1,
\end{equation}
noting that $|X_k|=p_k$. By {\bf Step 1}, we have that $F^k = (K^{\kk[X_k]}_{\bullet})_{(d_k^{p_k})}$ is an object in $\com_1^{\sym_{p_k}}$ for $1\leq k\leq s$, which admits a filtration with composition factors
\[\gr^l(F^k) = (D_{l+1}^k\simeq U_l^k) = E(S^{(p_k-l,1^l)})[-l]\text{ for }0\leq l\leq p_{k}-1,\]
where
\[U_l^k \simeq S^{(p_k-l,1^l)},\quad D_l^k \simeq S^{(p_k-l+1,1^{l-1})}.\]
We think of $(K_\bullet^{\kk[X_{s+1}]})_{(0^{p_{s+1}})}=\kk$ as an object in $\com_0^{\sym_{p_{s+1}}}$, and represent it by the Specht module $S^{(p_{s+1})}$. If we let $\sym_{\ul{p}}=\sym_{p_1} \times \cdots \times \sym_{p_s}\times\sym_{p_{s+1}}$ then we have
\begin{equation}\label{eq:K-mu-from-Fi}
(K_\bullet^R)_{\bm \mu}
= F^1\bt \cdots \bt F^s \bt S^{(p_{s+1})} \in \com_s^{\sym_{\ul{p}}}.
\end{equation}
For $(0^s)\leq \bmc=(c_1,\cdots,c_s)\leq(p_1-1,\cdots,p_s-1)$ we define
\begin{equation}\label{eq:def-E-mu-c}
\begin{aligned}
E^{\mu,\bmc}&=\gr^{c_1}(F^1)\bt\cdots\bt\gr^{c_s}(F^s)\bt S^{(p_{s+1})} \\
&=E\left(S^{(p_1-c_1,1^{c_1})}\bt\cdots\bt S^{(p_s-c_s,1^{c_s})}\bt S^{(p_{s+1})} \right)[-\bmc] \in \com_s^{\sym_{\ul{p}}}.
\end{aligned}
\end{equation}
Using (\ref{eq:K-mu-from-Fi}) and the discussion in Section~\ref{subsec:multi-cx}, we have that $(K_\bullet^R)_{\bm \mu}$ admits a filtration with composition factors
\[\gr^l(K_\bullet^R)_{\bm \mu} = \bigoplus_{|\bmc|=l} E^{\mu,\bmc},\text{ for }0\leq l\leq (p_1-1)+\cdots+(p_s-1).\]
Using (\ref{eq:def-K-mu}), we have $K_\bullet^\mu = \mathrm{Ind}_{\sym_{\ul{p}}}^{\sym_n}((K_\bullet^R)_{\bm \mu})$, and since induction is an exact functor, we have that $K^{\mu}_{\bullet}$ admits a filtration with
\[ \gr^l(K^{\mu}_{\bullet})=\bigoplus_{|\bmc|=l}\mathrm{Ind}_{\sym_{\ul{p}}}^{\sym_n}(E^{\mu,\bmc}) = \bigoplus_{|\bmc|=l}E\left( \mathcal S^{((p_1-c_1,1^{c_1}),\dots,(p_s-c_s,1^{c_s}),(p_{s+1}))} \right)[-\bmc],\]
where the last equality uses (\ref{eq:mcS-pi}), (\ref{eq:def-E-mu-c}), and the fact the construction of Boolean complexes in Example~\ref{ex:boolean} commutes with induction. If we define
\begin{equation}\label{eq:Lmu-c-boolean} L^{\mu,\bmc}_{\bullet} = E\left( \mathcal S^{((p_1-c_1,1^{c_1}),\dots,(p_s-c_s,1^{c_s}),(p_{s+1}))} \right)[-\bmc]
\end{equation}
for $(0^s)\leq \bmc\leq(p_1-1,\cdots,p_s-1)$,
then based on Corollaries~\ref{cor:filtrations} and~\ref{cor:G-filtrations} we get a decomposition
 \begin{equation}\label{eq:K-mu-sum} K^{\mu}_{\bullet}=\bigoplus_{(0^s)\leq \bmc\leq(p_1-1,\cdots,p_s-1)} L^{\mu,\bmc}_{\bullet}
 \end{equation}
 which is $\sym_n$-equivariant when $\kk$ has characteristic zero or $>n$.

Before explaining the proof of Theorem~\ref{thm:main}, it will be useful to analyze an example in order to illustrate the structure of $K^{\mu}_{\bullet}$.

\begin{example}\label{ex:K551}
 Suppose that $n=3$ and $\mu=(5,5,1)$. We have $s=2$, $p_1=2$, $p_2=1$, and $p_3=0$.
 Then $K^{\mu}_{\bullet}$ is $\ZZ^2$-complex (see (\ref{eq:decomp-Kl})) by 
 \eqref{eq:K-mu-from-Fi}, and using the $\bb{Z}^2$-grading on $K^{\mu}_{\bullet}$, we can picture the complex as
\begin{equation}\label{eq:K551-Z2}
\begin{gathered}
\xymatrix @C=2.5pc @R=0pc{
& & & & K^{\mu}_{(2,0)} \ar[dl] & \\
& & & K^{\mu}_{(1,0)} \ar[dl] & & K^{\mu}_{(2,1)} \ar[dl] \ar[ul] \\
K^{\mu}_{\bullet}: & 0 & K^{\mu}_{(0,0)} \ar[l] & & K^{\mu}_{(1,1)} \ar[ul] \ar[dl] & \\
& & & K^{\mu}_{(0,1)} \ar[ul] & & \\
}
\end{gathered}
\end{equation}
From~(\ref{eq:K-mu-sum}), we have a direct sum decomposition
 \[K^{\mu}_{\bullet} = L^{\mu,(0,0)}_{\bullet} \oplus L^{\mu,(1,0)}_{\bullet},\]
 where the summands are Boolean complexes. Using (\ref{eq:def-K-F}), we refine (\ref{eq:K551-Z2}) to 
\begin{equation}\label{eq:K551-sumL}
\begin{gathered}
\xymatrix @C=2.5pc @R=0pc{
& & & & *+[F]{\color{red}L^{\mu,(1,0)}_{\{1\}}} \ar[dl] & \\
& & & *+[F]{\begin{array}{c} {\color{red}L^{\mu,(1,0)}_{\varnothing}} \\ \oplus \\ {\color{blue}L^{\mu,(0,0)}_{\{1\}}} \end{array}} \ar[dl] & & *+[F]{\color{red}L^{\mu,(1,0)}_{\{1,2\}}} \ar[dl] \ar[ul] \\
K^{\mu}_{\bullet}: & 0 & *+[F]{\color{blue}L^{\mu,(0,0)}_{\varnothing}} \ar[l] & & *+[F]{\begin{array}{c} {\color{red}L^{\mu,(1,0)}_{\{2\}}} \\ \oplus \\ {\color{blue}L^{\mu,(0,0)}_{\{1,2\}}} \end{array}} \ar[ul] \ar[dl] & \\
& & & *+[F]{\color{blue}L^{\mu,(0,0)}_{\{2\}}} \ar[ul] & & \\
}
\end{gathered}
\end{equation}
where the blue terms come from $L^{\mu,(0,0)}_{\bullet}$, and the red ones from $L^{\mu,(1,0)}_{\bullet}$. Notice that $L^{\mu,\bmc}_F$ is a summand of $K^{\mu}_{\bmb}$ if and only if $\bmb=\bmc+\ee_F$. Each of the complexes $L^{\mu,(0,0)}_{\bullet}$ and  $L^{\mu,(1,0)}_{\bullet}$ is isomorphic up to a shift to some number of copies of the reduced simplicial complex of a $1$-dimensional simplex. As an $\sym_3$-representation, each of the blue modules is isomorphic to $\mc{S}^{((2),(1))}$, and each of the red modules is isomorphic to $\mc{S}^{((1,1),(1))}$.
\end{example}

We are now ready to prove our main result.

\begin{proof}[Proof of Theorem~\ref{thm:main}] Recall that $\Tor_i(I,\kk)$ can be computed as the $i$-th homology group of the subcomplex of $K_{\bullet}^R$ given by
\[ K_{\bullet}(I) = K_{\bullet}^R \oo_R I.\]
If we fix $\mu\in P_n$ then $\Tor_i(I,\kk)_{\langle \mu \rangle}$ is the homology of $\left(K_\bullet(I) \right)_{\langle \mu \rangle}$, which is a subcomplex of $K^{\mu}_{\bullet}$. We will describe $\left(K_\bullet(I) \right)_{\langle \mu \rangle}$ in relation to the decomposition~(\ref{eq:K-mu-sum}).

Consider any $\bmb\in\supp(K^{\mu}_{\bullet})$ (as defined in (\ref{eq:def-supp})), and note that $0\leq b_k\leq p_k$ for all $k=1,\cdots,s$. Using the notation (\ref{eq:mu-minus-c}) and (\ref{eq:def-Xk}), we write
\[\mu \setminus \bmb=\mu -\eee_{G_1 \cup G_2 \cup \cdots \cup G_s},\text{ for subsets }G_k \subseteq X_k\text{ with }|G_k|=b_k.\]
We then have 
\begin{equation}\label{eq:span-set-Kmu-b}
K_{\bmb}^\mu= \mathrm{span}_\kk \left\{ \sigma \left(x^{\mu \setminus \bmb} \cdot e_{G_1} \wedge \cdots \wedge e_{G_s} \right) \mid \sigma \in \sym_n\right\},
\end{equation}
where $e_G=e_{g_1}\wedge \cdots \wedge e_{g_m}$ for $G=\{g_1,\cdots,g_m\} \subseteq [n]$. Since $K_\bullet^R(I)$ is the subcomplex of $K_\bullet^R$ with
$$K_l^R(I) = \mathrm{span}_\kk \{x^{\bma} \cdot e_{i_1} \wedge \cdots \wedge e_{i_l} \mid x^{\bma} \in I,\ \{i_1,\dots,i_l\} \subseteq [n] \},$$
the equation (\ref{eq:span-set-Kmu-b}) tells us that an element of $K_{\bmb}^\mu$ appears in $K_\bullet^R(I)$
if and only if $\mu\setminus\bmb \in P(I)$, and in that case the whole $K_{\bmb}^\mu$ is contained in $K_\bullet^R(I)$. This shows that
\[
K_l^R(I)_{\langle \mu \rangle} = \bigoplus_{\substack{|\bmb|=l \\ \mu \setminus {\bmb} \in P(I)}} K_{\bmb}^\mu.
\]
Using (\ref{eq:K-mu-sum}) and the notation (\ref{eq:def-K-F}) as in Example~\ref{ex:K551}, it follows that $L^{\mu,\bmc}_F$ is a summand of $K_l^R(I)_{\langle \mu \rangle}$ if and only if $\mu\setminus(\bmc+\ee_F)\in P(I)$, which by (\ref{eq:def-Del-mu-c-I}) is equivalent to $F$ being a face of $\Delta^{\mu,\bm c}(I)$. If we consider the subcomplex $L^{\mu,\bmc}_{\bullet}(I)\subseteq L^{\mu,\bmc}_{\bullet}$ defined by
\[L^{\mu,\bmc}_{l+|\bmc|}(I) = \bigoplus_{\substack{F\in\Delta^{\mu,\bm c}(I),\\ |F|=l}} L^{\mu,\bmc}_F,\]
then it follows that
\begin{equation}\label{eq:KImu-sum-LI}
( K_\bullet^R(I))_{\langle \mu \rangle} = \bigoplus_{\bm 0 \leq \bm c \leq \bf p(\mu)} L_\bullet^{\mu,\bm c}(I),
\end{equation}
and moreover, we have from (\ref{eq:Lmu-c-boolean}) that
\begin{equation}\label{eq:LI-as-chain-Delta}
L_\bullet^{\mu,\bm c}(I) \cong \widetilde C_{\bullet+1+|\bm c|}(\Delta^{\mu,\bm c}(I)) \otimes_\kk \mathcal S^{((p_1-c_1,1^{c_1}),\dots,(p_s-c_s,1^{c_s}),(p_{s+1}))}.
\end{equation}
Combining (\ref{eq:KImu-sum-LI}) with (\ref{eq:LI-as-chain-Delta}) and taking homology yields the desired description of $\Tor_i(I,\kk)_{\langle \mu \rangle}$, concluding the proof. 
\end{proof}

We end this section by illustrating the proof of Theorem~\ref{thm:main} with an example.

\begin{example}\label{ex:K551-contd}
 We continue with the notation in Example~\ref{ex:K551}, and consider the ideal $I=\langle (4,1,1),(5,2,0)\rangle$. When considering the subcomplex $(K_\bullet(I))_{\langle \mu \rangle}\subseteq K^{\mu}_{\bullet}$, the term $K^{\mu}_{(2,1)}$ disappears since $\mu\setminus (2,1)=(4,4,0) \not \in P(I)$. We get from (\ref{eq:K551-sumL})
\[
\xymatrix @C=2.5pc @R=0pc{
& & & & *+[F]{\color{red}L^{\mu,(1,0)}_{\{1\}}(I)} \ar[dl] \\
& & & *+[F]{\begin{array}{c} {\color{red}L^{\mu,(1,0)}_{\varnothing}(I)} \\ \oplus \\ {\color{blue}L^{\mu,(0,0)}_{\{1\}}}(I) \end{array}} \ar[dl] & \\
K^{\mu}_{\bullet}(I): & 0 & *+[F]{\color{blue}L^{\mu,(0,0)}_{\varnothing}(I)} \ar[l] & & *+[F]{\begin{array}{c} {\color{red}L^{\mu,(1,0)}_{\{2\}}(I)} \\ \oplus \\ {\color{blue}L^{\mu,(0,0)}_{\{1,2\}}(I)} \end{array}} \ar[ul] \ar[dl] \\
& & & *+[F]{\color{blue}L^{\mu,(0,0)}_{\{2\}}(I)} \ar[ul] & \\
}
\]
The red complex $L^{\mu,(1,0)}_{\bullet}(I)$ is then the $\sym_3$-module $\mc{S}^{((1,1),(1))}$ tensored with the reduced chain complex of two points, while the blue complex $L^{\mu,(0,0)}_{\bullet}(I)=L^{\mu,(0,0)}$ remains acyclic. It follows as noted in the introduction that
\[\Tor_i(I,\kk)_{\langle (5,5,1) \rangle} = \begin{cases}
\mc{S}^{((1,1),(1))} & \text{if }i=2;\\
0 & \text{otherwise}.
\end{cases}\]
\end{example}

\section{The $\sym_n$-invariant part of the Betti table}
\label{sec:sn-inv-Betti}

The goal of this section is to give a quick application of Theorem~\ref{thm:main} and the Nerve Theorem, computing the $\sym_n$-invariant part of $\Tor_i(I,\kk)$ when $I$ is an $\sym_n$-invariant monomial ideal, and $\kk$ is a field of characteristic zero or $>n$. More precisely, we show the following.

\begin{theorem}
\label{thm:invariantpart}
Let $\lambda^1,\dots,\lambda^r \in P_n$.
Then, for any $\mu \in P_n$, one has
\begin{align} \label{eq:6.1-1}
\gamma_i^{\mu,\bm 0}(\langle \lambda^1,\cdots, \lambda^r \rangle_{\sym_n})=\dim_\kk \Tor_i((x^{\lambda^1},\cdots,x^{\lambda^r}),\kk)_{\bm \mu} \ \mbox{ for all }i.
\end{align}
In particular, if $\mathrm{char}(\kk)=0$ or $\mathrm{char}(\kk)>n$, then
\begin{align} \label{eq:6.1-2}
    \Tor_i \big(\langle \lambda^1,\cdots, \lambda^r \rangle_{\sym_n},\kk \big)^{\sym_n} \cong \Tor_i\big((x^{\lambda^1},\cdots,x^{\lambda^r}),\kk\big) \ \mbox{ for all }i.
\end{align}
\end{theorem}

\begin{proof}
Let $I=\langle \lambda^1,\cdots, \lambda^r \rangle_{\sym_n}$ and $J=(x^{\lambda^1},\cdots,x^{\lambda^r})$.
For a subset $\Lambda\subset P_n$ we define the partition $\lcm(\Lambda)\in P_n$ by
 \[\lcm(\Lambda)_i = \max\{\lambda_i \mid \lambda\in\Lambda\}.\]
Also, for a subset $G\subseteq [r]$ we write
\[\lcm(G) = \lcm(\{\lambda^i \mid i\in G\}).\]
Consider the simplicial complex
\[ X_{<\mu} = \{ G\subseteq[r] \mid \lcm(G) < \mu\}.\]
It follows from \cite[Theorem 1.11]{BS} (see also the proof of \cite[Theorem 2.1]{GPW}) that
\[\Tor_i(J,\kk)_{\bm \mu} \cong \widetilde H_{i-1}(X_{<\mu}).
\]
Then it follows that in order to prove \eqref{eq:6.1-1}, it suffices to show that the complexes $\Delta^{\mu,\bm 0}(I)$ and $X_{<\mu}$ are homotopy equivalent, which we do next.

If $\mu=(d_1^{p_1},\dots,d_s^{p_s},0^{p_{s+1}}) \in P_n$ with $d_1>\cdots>d_s>0$, then we have
\[ \lambda<\mu \Longleftrightarrow \lambda\leq\mu\setminus\ee_i\text{ for some }i=1,\cdots,s.\]
It follows that the facets of $X_{<\mu}$ are 
\[ G_i = \{j\in[r] \mid \lambda^j\leq \mu \setminus\ee_i\}\text{ for }i=1,\cdots,s.\]
The Nerve Theorem (see for instance \cite[Theorem 10.6]{Bj}) implies that $X_{<\mu}$ is homotopy equivalent to the \defi{nerve} of $G_1,\cdots,G_s$, which is the simplicial complex 
$$\mathcal N(G_1,\dots,G_s)=\left\{F \subset [s] \mid \bigcap_{i \in F} G_i \ne \varnothing \right\}.$$
We note that $\bigcap_{i \in F} G_i \ne \varnothing$ is equivalent to the fact that for some $\lambda^j$ we have
\[ \lambda^j \leq \mu\setminus \ee_i \text{ for all }i\in F.\]
This is further equivalent to $\lambda^j \leq \mu\setminus \ee_F$, which shows that
\[\bigcap_{i \in F} G_i \ne \varnothing \Longleftrightarrow \mu\setminus \ee_F \in P(I).\]
It follows from \eqref{eq:def-Del-mu-c-I} that $\mathcal N(G_1,\dots,G_s)=\Delta^{\mu,\bm 0}(I)$, so $X_{<\mu}$ is homotopy equivalent to $\Delta^{\mu,\bm 0}(I)$, proving \eqref{eq:6.1-1}.

 We now assume that $\mathrm{char}(\kk)=0$ or $\mathrm{char}(\kk)>n$ and prove \eqref{eq:6.1-2}.
Using the Taylor resolution of $J$ \cite[\S 7.1]{HH}, we have that if $\Tor_i(J,\kk)_{\bma}\neq 0$ for some $\bma\in\bb{Z}^n_{\geq 0}$ then $\bma = \lcm(\Lambda)$ for some subset $\Lambda\subseteq\{\lambda^1,\dots, \lambda^r\}$, and in particular $\bma\in P_n$.
Thus, to prove \eqref{eq:6.1-2}, it is then enough to show that
\[
\Tor_i(I,\kk)^{\sym_n}_{\langle \mu \rangle} \cong \Tor_i(J,\kk)_{\bm \mu} \text{ for all }\mu\in P_n.
\]
It follows from the Littlewood--Richardson rule (see e.g., \cite[Theorem 4.9.14]{Sa}) that
$$\left( \mathcal S^{((p_1,1^{q_1}),\dots,(p_s,1^{q_s}))}\right)^{\sym_n} \cong
\begin{cases}
\kk & \mbox{ if } q_1=\cdots=q_s=0;\\
0 & \mbox{ otherwise.}
\end{cases}
$$
Thus Theorem \ref{thm:main} and \eqref{eq:6.1-1} imply the desired isomorphism
\[
\Tor_i(I,\kk)^{\sym_n}_{\langle \mu \rangle}
\cong \widetilde H_{i-1}(\Delta^{\mu,\bm 0}(I))
\cong \Tor_i(J,\kk)_{\bm \mu}.\qedhere
\]
\end{proof}

\section{Primary decomposition and extremal Betti numbers}
\label{sec:prim-dec-extr-Betti}

The goal of this section is to describe a primary decomposition for any $\sym_n$-invariant monomial ideal $I$, and to study the extremal Betti numbers of $I$. As an application, we recover using Theorem~\ref{thm:main} the formulas from \cite{Ra} for the Castelnuovo--Mumford regularity, and for the projective dimension of~$I$. To formulate our results, we consider the set of \defi{extended partitions}
$$P_n^\infty=\{(\lambda_1,\dots,\lambda_n) \in (\ZZ_{\geq 0} \cup \{\infty\})^n \mid \lambda_1 \geq \cdots \geq \lambda_n \geq 0\}$$
where $\infty \geq a$ for any $a \in \ZZ_{\geq 0} \cup \{\infty\}$, for which a partial order is constructed in Section~\ref{subsec:max-dual-gens}. In Section~\ref{subsec:prim-dec} we determine a finite subset $\Lambda^*(I)\subset P_n^{\infty}$ describing a natural primary decomposition of $I$, and refer to $\Lambda^*(I)$ as the set of \defi{dual generators} of $I$. If $\rho=(\infty^{p_0},d_1^{p_1},\cdots,d_s^{p_s})\in P_n^{\infty}$, with $\infty>d_1>\cdots>d_s\geq 0$, then we write 
\begin{equation}\label{eq:def-l-rho}
\ell(\rho)=p_0
\end{equation}
for the number of $\infty$ terms in $\rho$, and let
\begin{equation}\label{eq:def-tilde-rho}
\widetilde{\rho} = ((d_1+1)^{p_0+p_1},(d_2+1)^{p_2},\cdots,(d_s+1)^{p_s}).
\end{equation}

In analogy with the multigraded version of extremal Betti numbers from \cite[p.~507]{BCP}, we say that a pair $(i,\lambda) \in [n] \times P_n$ is \defi{extremal} in the Betti table of $R/I$ if
\begin{itemize}
\item[(a)] $\Tor_i(R/I,\kk)_{\langle \lambda \rangle} \ne 0$, and
\item[(b)] $\Tor_j(R/I,\kk)_{\langle \mu \rangle}=0$ for all $j \geq i$ and $\mu \gneq \lambda$ with $|\mu|-j \geq |\lambda|-i$.
\end{itemize}
If $(i,\lambda)$ is extremal, the \defi{extremal Betti number} $\beta_{i,\lambda}(I)$ is $\dim_{\kk}\Tor_i(R/I,\kk)_{\bm \lambda}$ (which is equal to $\beta_{i,\sigma\cdot\lambda}(I)$ for all $\sigma\in\sym_n$). The main result of this section is the following.

\begin{theorem}
\label{thm:extremal}
For any $\sym_n$-invariant monomial ideal $I$, we have
\begin{align*}
&\{(i,\lambda) \in [n] \times P_n \mid (i,\lambda) \mbox{ is an extremal pair in the Betti table of $R/I$}\}
\\
&= \{ \left(n-\ell(\rho),\widetilde \rho \right )\mid \rho\in\Lambda^*(I) \mbox{ is a maximal dual generator of }I\}.
\end{align*}
Moreover, if $(i,\lambda)$ is the extremal pair associated to $\rho=(\infty^{p_0},d_1^{p_1},\cdots,d_s^{p_s})$, then the corresponding extremal Betti number is $\beta_{i,\lambda}={p_0+p_1-1\choose p_0}$.
\end{theorem}

We prove Theorem~\ref{thm:extremal} in Section~\ref{subsec:extr-Betti} using a reformulation of Theorem~\ref{thm:main} via Alexander duality, which is explained in Section~\ref{subsec:main-thm-Alex-dual}. In Section~\ref{subsec:extr-Betti-Ext} we discuss the relationship between our results and the combinatorics used in \cite{Ra}, and explain the relation of Theorem~\ref{thm:extremal} to the study of $\Ext$ modules.

\subsection{Primary decompositions of $\sym_n$-invariant monomial ideals}
\label{subsec:prim-dec}

We begin by recalling a canonical primary decomposition for a monomial ideal \cite[Lemma~3.1]{HH}.

\begin{lemma}
\label{lem:6.1}
Every monomial ideal $I$ of $R$ has a presentation
\begin{align}
\label{6-1}
I=Q_1 \cap Q_2 \cap \cdots \cap Q_r,
\end{align}
where each $Q_i$ is an ideal of the form $(x_{i_1}^{a_1},\dots,x_{i_k}^{a_k})$.
Moreover, such a presentation is unique if it is irredundant, i.e., if none of the ideals $Q_i$ can be omitted from \eqref{6-1}.
\end{lemma}

To describe the irredundant presentation \eqref{6-1} for an $\sym_n$-invariant monomial ideal, we define for each $\mu=(\infty,\dots,\infty,\mu_k,\dots,\mu_n) \in P_n^\infty$ with $\mu_k < \infty$, the ideal
\[
Q_\mu=\bigcap_{\sigma \in \sym_n} \sigma(x_k^{\mu_k+1},\dots,x_n^{\mu_n+1}).
\]
If $I \subset R$ is an $\sym_n$-invariant monomial ideal, its irredundant presentation \eqref{6-1} is preserved by the $\sym_n$-action. Therefore, if $\sigma \in \sym_n$ then $\sigma(Q_k)=Q_l$ for some $1 \leq l \leq r$.
This fact and Lemma \ref{lem:6.1} imply the following.

\begin{lemma}
\label{lem:decomposition}
Let $I \subset R$ be an $\sym_n$-invariant monomial ideal.
Then there are unique elements $\mu^1,\cdots,\mu^t \in P_n^\infty$ such that
\begin{align}
\label{6-2}
I= Q_{\mu^1} \cap \cdots \cap Q_{\mu^t}
\end{align}
and none of the ideals $Q_{\mu^k}$ can be omitted in the above presentation.
\end{lemma}

We call \eqref{6-2} the \defi{irredundant decomposition} of $I$, and call $\mu^1,\dots,\mu^t$ the \defi{dual generators} of $I$. We write
\begin{equation}\label{eq:def-lam-star}
\Lambda^*(I)=\{\mu^1,\dots,\mu^t\},
\end{equation}
and note that the condition that \eqref{6-2} is irredundant implies that
\begin{equation}\label{eq:dual-gens-incomp}
    \mu^i \not\leq \mu^j \text{ for }1\leq i\neq j\leq t.
\end{equation}

\begin{remark}
\label{rem:decomposition}
If $\mu=(\infty^{p_0},d_1^{p_1},\dots,d_m^{p_m})$ with $\infty >d_1> \cdots >d_m \geq 0$,
then
$$Q_{\mu}=\big\langle \big((d_1+1)^{p_0+1}\big),\big((d_2+1)^{p_0+p_1+1}\big),\dots,\big((d_m+1)^{p_0+p_1+\cdots+p_{m-1}+1}\big)\big\rangle _{\sym_n}.$$
The ideals $Q_\mu$ are therefore the $\sym_n$-invariant ideals generated by a set of rectangular partitions. Combining the formula for $Q_{\mu}$ with
\begin{align*}
\langle (d_1^{p_1},\dots,d_s^{p_s})\rangle_{\sym_n}=
\langle (d_1^{p_1})\rangle_{\sym_n} \cap
\langle (d_2^{p_1+p_2})\rangle_{\sym_n}
\cap \cdots \cap
\langle (d_s^{p_1+p_2+ \cdots+p_s})\rangle_{\sym_n}
\end{align*}
provides a way to compute the presentation \eqref{6-2}.
For example, we have
\begin{align*}
\langle (4,1,1),(5,2,0)\rangle_{\sym_3}&= \langle (4,0,0),(5,2,0) \rangle_{\sym_3} \cap \langle (1,1,1),(5,2,0) \rangle_{\sym_3}\\
&= \langle (4,0,0)\rangle_{\sym_3} \cap \langle (1,1,1),(5,0,0)\rangle_{\sym_3} \cap \langle (1,1,1),(2,2,0)\rangle_{\sym_3}\\
&=Q_{(3,3,3)} \cap Q_{(4,4,0)} \cap Q_{(\infty,1,0)},
\end{align*}
where for the first two equalities we used 
\[
\langle (4,1,1) \rangle_{\sym_3}= \langle (4,0,0) \rangle_{\sym_3} \cap \langle (1,1,1) \rangle_{\sym_3}\text{ and }
\langle (5,2,0) \rangle_{\sym_3}= \langle (5,0,0) \rangle_{\sym_3} \cap \langle (2,2,0) \rangle_{\sym_3}.
\]
We conclude that $\Lambda^*(\langle (4,1,1),(5,2,0)\rangle_{\sym_3})=\{(3,3,3),(4,4,0),(\infty,1,0)\}$.
\end{remark}

To shed more light on the set $\Lambda^*(I)$, we define
$$O(I)=P_n \setminus P(I)=\{\lambda \in P_n \mid x^\lambda \not \in I\},$$
which is the set of all partitions that are not in $I$. The irredundant decomposition \eqref{6-2} is related to $O(I)$ as follows. For $\mu=(\infty,\dots,\infty,\mu_k,\dots,\mu_n) \in P_n^\infty$, let
$$O_\mu=\{ \lambda \in P_n \mid \lambda \leq \mu\}.$$
One can check that
$$O_\mu=O(Q_{\mu}),$$
hence Lemma~\ref{lem:decomposition} implies that for any $\sym_n$-invariant monomial ideal $I \subset R$, one has
\begin{align}
\label{6-3}
O(I)= \bigcup_{\mu \in \Lambda^*(I)} O_\mu.
\end{align}
Moreover, the decomposition \eqref{6-3} is irredundant (that is, no $O_{\mu}$ can be omitted). 

\subsection{A reformulation of Theorem~\ref{thm:main} via Alexander duality}
\label{subsec:main-thm-Alex-dual}

With the notation in Section~\ref{sec:main-thm}, we define for $\mu=(d_1^{p_1},\dots,d_s^{p_s},0^{p_{s+1}})$ and $\bm c \leq \bf p(\mu)$ the simplicial complex
\begin{equation}\label{eq:def-Gamma-mu-c}
\Gamma^{\mu,\bm c}(I)=\{ F \subseteq [s] \mid \mu \setminus (\bm c + \ee_{[s]\setminus F}) \in O(I)\}.
\end{equation}
In other words,
$\Gamma^{\mu,\bm c}(I)=\{F \subseteq [s] \mid [s] \setminus F \not \in \Delta^{\mu,\bm c}(I)\}$, that is, $\Gamma^{\mu,\bm c}(I)$ is the \defi{Alexander dual} of $\Delta^{\mu,\bm c}$. We have by \cite[Lemma 5.5.3]{BH} that
$$\widetilde H_{i-2}(\Delta^{\mu,\bm c}(I)) \cong \widetilde H_{s-i-1} (\Gamma^{\mu,\bm c}(I)).$$
Using this isomorphism, Theorem \ref{thm:main} can be rewritten as follows.

\begin{theorem}
\label{thm:main6}
Let $\mu=(d_1^{p_1},\dots,d_s^{p_s},0^{p_{s+1}}) \in P_n$ and let $I \subset R$ be an $\sym_n$-invariant monomial ideal.
We have an isomorphism of $\kk$-vector spaces
$$\Tor_i(R/I,\kk)_{\langle \mu \rangle}
\cong \bigoplus_{\bm 0 \leq\bmc\leq \bf p(\mu)} \left ( \mathcal S^{((p_1-c_1,1^{c_1}),\dots,(p_s-c_s,1^{c_s}),(p_{s+1}))} \right)^{\dim_\kk \widetilde H_{s-i-1+|\bm c|}(\Gamma^{\mu,\bm c}(I))},$$
which is in addition an isomorphism of $\sym_n$-modules when $\mathrm{char}(\kk)=0$ or $\mathrm{char}(\kk)>n$.
\end{theorem}

\subsection{Maximal dual generators}
\label{subsec:max-dual-gens}

In what follows we introduce a partial order on the set of dual generators $\Lambda^*(I)$ in \eqref{eq:def-lam-star}, and explain how the maximal elements of $\Lambda^*(I)$ contribute to the Betti numbers of $R/I$. For $\mu=(\infty,\dots,\infty,\mu_k,\dots,\mu_n) \in P_n^\infty$ with $\mu_k \ne \infty$, we define $\ell(\mu)=k-1$ as in (\ref{eq:def-l-rho}), define $\widetilde{\mu}$ as in (\ref{eq:def-tilde-rho}), and let
\[
\mu^+=(\mu_k+1,\dots,\mu_k+1,\mu_k,\dots,\mu_n).
\]
We note that $\mu^+$ is obtained from $\mu$ by replacing $\infty$ with $\mu_k+1$, and that
\[
\widetilde{\mu} = \mu^++\eee_k+\cdots+\eee_n.
\]

We define the partial order $\preccurlyeq$ on $\Lambda^*(I)$ by $\mu \preccurlyeq \rho$ if 
\begin{equation}\label{eq:def-prec-on-LamI}
\widetilde{\mu} \leq \widetilde{\rho}\text{ and }\ell(\mu)-\ell(\rho) \leq |\widetilde{\rho}| - |\widetilde{\mu}|.
\end{equation}
Using the fact that
\begin{equation}\label{eq:size-mu+-mutilde}
    |\mu^+| = |\widetilde{\mu}| - (n-\ell(\mu)),
\end{equation}
we can rewrite the conditions \eqref{eq:def-prec-on-LamI} as
\begin{equation}\label{eq:equiv-prec-on-LamI}
\widetilde{\mu} \leq \widetilde{\rho}\text{ and }|\mu^+| \leq |\rho^+|.
\end{equation}
We write $\mu\prec\rho$ if $\mu\preccurlyeq\rho$ and $\mu\neq\rho$. We let $\Lambda^*_{\max}(I)\subseteq\Lambda^*(I)$ denote the subset of maximal elements with respect to $\preccurlyeq$, and call them \defi{maximal dual generators} of $I$.

\begin{lemma}
\label{lem:order}
Let $I$ be an $\sym_n$-invariant monomial ideal.
\begin{itemize}
\item[(i)] If $\mu,\rho \in \Lambda^*(I)$ satisfy $\mu \prec \rho$ then $\ell(\mu)>\ell(\rho)$.
\item[(ii)] If $\rho \in \Lambda^*_{\max}(I)$ then $\rho^+ \not\in O_{\mu}$ for any $\mu \in \Lambda^*(I) \setminus \{\rho\}$.
\end{itemize}
\end{lemma}

\begin{proof}
(i) Write $\mu=(\infty,\dots,\infty,\mu_k,\dots,\mu_n)\prec\rho=(\infty,\dots,\infty,\rho_l,\dots,\rho_n)$, and
suppose by contradiction that $\ell(\mu) \leq  \ell(\rho)$, or equivalently, that $k\leq l$.
The condition $\widetilde{\mu}\leq\widetilde{\rho}$ implies $\mu_m\leq \rho_m$ for all $l\leq m\leq n$.
Since for $1\leq m<l$ we have $\mu_m\leq\infty=\rho_m$, this shows that $\mu\leq\rho$,
contradicting \eqref{eq:dual-gens-incomp}.

(ii) 
Let $\mu=(\infty,\dots,\infty,\mu_k,\dots,\mu_n)$, $\rho=(\infty,\dots,\infty,\rho_l,\dots,\rho_n)$, and suppose that $\rho \in \Lambda^*_{\max}(I)$ and $\mu \in \Lambda^*(I) \setminus \{\rho\}$. If $\mu_m<\rho_m$ for some $l \leq m \leq n$, then $\rho^+ \not \leq \mu$, hence $\rho^+\not\in O_{\mu}$, as desired. We may therefore assume that $\mu_m \geq \rho_m$ for all $l \leq m \leq n$.
By \eqref{eq:dual-gens-incomp} we have $\rho \not \leq \mu$, hence $\mu_{l-1} \ne \infty$ and thus $k<l$. If $\mu_{l-1}\geq\rho_l+1$ then we have $\mu^+\geq\rho^+$ and $\widetilde{\mu}\geq\widetilde{\rho}$, which implies by (\ref{eq:equiv-prec-on-LamI}) that $\mu\succ\rho$, contradicting the maximality of $\rho$. It follows that $\mu_{l-1} < \rho_l+1$, so $\rho^+ \not \leq \mu$, concluding the proof.
\end{proof}

Maximal dual generators have the following contributions to Betti numbers.

\begin{lemma}
\label{lem:maximal}
Let $\rho=(\infty^{p_0},d_1^{p_1},\dots,d_{s}^{p_s}) \in \Lambda^*_{\max}(I)$, with $\infty>d_1>\cdots>d_s \geq 0$,
and let $\bmc=(p_1-1,\dots,p_s-1)$. We have $\Gamma^{\widetilde \rho,\bm c} = \{ \varnothing\}$, and
$$\beta_{n-\ell(\rho),\widetilde{\rho}}(R/I) = {p_0+p_1-1\choose p_0}.$$
\end{lemma}

\begin{proof}
We observe that $\widetilde{\rho}$ is as in (\ref{eq:def-tilde-rho}), and 
$$\widetilde \rho \setminus (\bm c + \ee_{[s]})=\big((d_1+1)^{p_0},d_1^{p_1},d_2^{p_2}\dots,d_s^{p_s}\big)=\rho^+.$$
Since $\rho^+ \leq \rho$, it follows from \eqref{6-3} that $\rho^+\in O(I)$, so $\varnothing \in \Gamma^{\widetilde \rho,\bm c}(I)$ by \eqref{eq:def-Gamma-mu-c}. To prove that $\Gamma^{\widetilde \rho,\bm c} = \{ \varnothing\}$, it then suffices to check that $\{i\} \not \in \Gamma^{\widetilde \rho,\bm c}(I)$ for $i \in [s]$.

We fix $i \in [s]$ and note that 
\[\widetilde \rho \setminus (\bm c + \ee_{[s]\setminus\{i\}})=\rho^+ +\eee_k\text{ for some }k>p_0,\]
hence $\widetilde \rho \setminus (\bm c + \ee_{[s]\setminus\{i\}}) \not \in O_\rho$. Since $\rho$ is maximal, we have by Lemma~\ref{lem:order}(ii) that $\rho^+ \not \in O_\mu$ for all $\mu \in \Lambda^*(I)$ with $\mu \ne \rho$. We get from \eqref{6-3} that $\widetilde \rho \setminus (\bm c + \ee_{[s]\setminus\{i\}}) \not \in O(I)$, hence $\{i\} \not \in \Gamma^{\widetilde \rho,\bm c}(I)$ by \eqref{eq:def-Gamma-mu-c}, as desired.

If we let $i=n-\ell(\rho)=n-p_0$ and $\mu=\widetilde{\rho}$ in Theorem~\ref{thm:main6}, then we have
\[s-i-1+|\bmc| = s-(n-p_0)-1+(n-p_0-s) = -1,\]
and
\[\dim_\kk \widetilde H_{s-i-1+|\bm c|}(\Gamma^{\mu,\bm c}(I)) = \dim_\kk \widetilde{H}_{-1}(\{\varnothing\}) = 1.\]
It follows from Theorem~\ref{thm:main6} that
\[\Tor_{n-\ell(\rho)}(R/I,\kk)_{\langle \widetilde{\rho}\rangle}=
\mc{S}^{((p_0+1,1^{p_1-1}),(1^{p_2}),\cdots,(1^{p_s}))} 
\]
Restricting to the multidegree $\widetilde{\rho}$, and using the fact that each of the Specht modules $S^{(1^{p_i})}$ has dimension $1$, it follows that
\[
\beta_{n-\ell(\rho),\widetilde{\rho}}(R/I) = \dim_{\kk} S^{(p_0+1,1^{p_1-1})} = {p_0+p_1-1\choose p_0},
\]
where the last equality follows from the Hook Length Formula \cite[\textsection~20]{James} (or a direct count of the standard tableaux in $\mathrm{Tab}((p_0+1,1^{p_1-1}))$).
\end{proof}

\begin{example}
\label{ex:6.7}
If we let $I=\langle (4,1,1),(5,2,0)\rangle_{\sym_3}$, then as seen in Remark~\ref{rem:decomposition}, we have
$\Lambda^*(I)=\{ (3,3,3),(4,4,0),(\infty,1,0)\}$.
The maximal dual generators of $I$ are then $(3,3,3)$ and $(4,4,0)$.
Since $\ell((3,3,3))=\ell((4,4,0))=0$, Lemma~\ref{lem:maximal} implies that they contribute to $\Tor_3(R/I,\kk)_{\langle (4,4,4)\rangle}$ and $\Tor_3(R/I,\kk)_{\langle (5,5,1) \rangle}$ respectively. More precisely, we have
\[ \beta_{3,(4,4,4)}(R/I) = \beta_{3,(5,5,1)}(R/I) = 1,\]
and
\[\Tor_3(R/I,\kk)_{\langle (4,4,4)\rangle} = \Tor_3(R/I,\kk)_{(4,4,4)}\]
 is $1$-dimensional, while
\[\Tor_3(R/I,\kk)_{\langle (5,5,1)\rangle} = \Tor_3(R/I,\kk)_{(5,5,1)} \oplus \Tor_3(R/I,\kk)_{(5,1,5)} \oplus \Tor_3(R/I,\kk)_{(1,5,5)}\]
is $3$-dimensional (see Example~\ref{ex:1.2}). We note that $(3,(4,4,4))$ and $(3,(5,5,1))$ are all the extremal pairs in the Betti table of $R/I$. 
\end{example}

\subsection{Extremal Betti numbers, regularity and projective dimension.}
\label{subsec:extr-Betti}

Our next goal is to give a proof of Theorem~\ref{thm:extremal}, and to derive formulas for the regularity and projective dimension of an $\sym_n$-invariant monomial ideal $I$. We begin by establishing some preliminary results.

We say that a simplicial complex $\Delta$ is a \defi{cone with apex $v$} if for every face $F \in \Delta$, one has $F \cup \{v\} \in \Delta$.
If $\Delta$ is a cone then it is contractible, and $\widetilde H_i(\Delta)=0$ for all $i$. We will need the following slight generalization of this fact.

\begin{lemma}
\label{lem:cone}
If $\Delta$ is a simplicial complex on the set $[n]$, with $\widetilde H_{i-1}(\Delta) \ne 0$, then there exists a facet $F$ of $\Delta$ with $ 1 \not \in F$ and $|F|\geq i$.
\end{lemma}

\begin{proof}
Let $\Delta'$ be the simplicial complex whose facets are the facets of $\Delta$ of dimension $\geq(i-1)$. Since $\Delta'$ and $\Delta$ have the same faces of dimension $\geq(i-1)$, it follows that $\widetilde H_{i-1}(\Delta') \cong \widetilde H_{i-1}(\Delta) \ne 0$, and in particular $\Delta'$ is not a cone with apex $1$. We get that $\Delta'$ contains a facet $F$ with $1 \not \in F$, which is then a facet of $\Delta$ with $|F| \geq i$.
\end{proof}

We next record two more technical statements before proving Theorem~\ref{thm:extremal}.

\begin{lemma}\label{lem:claim}
 Let $\mu=(d_1^{p_1},\dots,d_s^{p_s},0^{p_{s+1}})$, and define $\bf p(\mu)$ as in \eqref{eq:def-p-mu}, and $\Gamma^{\mu,\bm c}(I)$ as in \eqref{eq:def-Gamma-mu-c}, for some $\bmc\leq\bf p(\mu)$. If $F$ is a facet of $\Gamma^{\mu,\bm c}(I)$ with $1\not\in F$, we let
 $$\mu'=\mu \setminus (\bmc + \ee_{[s]\setminus F})=(\mu_1',\dots,\mu_n'),$$
 and define
 $$h=\min\{i \mid \mu_i' \ne \mu_1\}.$$
 If $\rho =(\infty,\dots,\infty,\rho_l,\dots,\rho_n) \in \Lambda^*(I)$ with $\rho_l\neq\infty$ and $\mu' \leq \rho$, then we have $h \geq l$ and $\mu_1 \leq \rho_l+1$.
\end{lemma}

\begin{proof}
 Since $1 \not \in F$, the first entry of $\bm c+ \ee_{[s] \setminus F}$ is positive. By \eqref{eq:mu-minus-c}, we have
$$\mu_h'+1 =\mu_h =\mu_1.$$
If $h < l$ then $\rho_h=\infty$, and using $\mu'\leq\rho$ we obtain
$$\mu \setminus (\bm c + \ee_{[s]\setminus (F \cup \{ 1\})})= \mu \setminus (\bm c + \ee_{[s]\setminus F}-\ee_1)=\mu'+\eee_h \leq \rho.$$
This shows that $\mu \setminus (\bm c+\ee_{[s]\setminus (F \cup \{1\})}) \in O_\rho \subset O(I)$, so $F\cup\{1\}\in\Gamma^{\mu,\bm c}(I)$ by \eqref{eq:def-Gamma-mu-c}, contradicting the fact that $F$ was a facet. It follows that $h\geq l$, and since $\mu' \leq \rho$, we have $\mu'_h \leq \mu_l' \leq \rho_l$. We get $\mu_1=\mu'_{h}+1 \leq \rho_l+1$, concluding the proof.
\end{proof}

For the next result we recall the definition of $s(\mu)$ from \eqref{eq:def-p-mu}.

\begin{lemma}
\label{lem:6.8}
If $\widetilde H_{i-1}(\Gamma^{\mu,\bmc}(I) ) \ne 0$  for some $\mu \in P_n$ and $\bmc \leq \bf p(\mu)$,
then there exists $\rho\in\Lambda^*_{\max}(I)$ such that
\begin{itemize}
\item[(i)] $|\bmc| +s(\mu)-i \leq n-\ell(\rho)$,
\item[(ii)] $|\mu|-(|\bmc| +s(\mu)-i) \leq |\rho^+|$, and
\item[(iii)] $\mu \leq \widetilde \rho$.
\end{itemize}
\end{lemma}

\begin{proof}
We note that each of $n-\ell(\rho)$, $|\rho^+|$, $\widetilde \rho$, increases as $\rho$ increases with respect to the order $\preccurlyeq$ by Lemma \ref{lem:order}(i). It follows that it is enough to find $\rho\in\Lambda^*(I)$ satisfying (i)--(iii).

By Lemma~\ref{lem:cone}, there is a facet $F$ of $\Gamma^{\mu,\bm c}(I)$ such that $1 \not \in F$ and $|F| \geq i$.
We let $\mu=(d_1^{p_1},\dots,d_s^{p_s},0^{p_{s+1}})$, so that $s=s(\mu)$, and define $\mu'$ and $h$ as in Lemma~\ref{lem:claim}. Since $F \in \Gamma^{\mu,\bm c}(I)$, we have $\mu' \in O(I)$, so there exists an element $\rho \in \Lambda^*(I)$ such that $\mu' \leq \rho$. We prove that $\rho$ satisfies conditions (i)--(iii).

We first note that $h=p_1-c_1$, and that $n=p_1+\cdots+p_s+p_{s+1}$. Combining these observations with the assumption $\bmc\leq\bf p(\mu)$, we get
$$n-|\bm c +\ee_{[s]}|\geq\sum_{j=1}^s(p_j-c_j-1)\geq p_1-c_1-1 =h-1.$$
Since $\mu'\leq\rho$, we know from Lemma~\ref{lem:claim} that $h\geq\ell(\rho)+1$, so
$$|\bm c| + s-i \leq |\bm c+\ee_{[s]}| \leq n-(h-1) \leq n-\ell(\rho),$$
proving (i). The conclusion of Lemma~\ref{lem:claim} implies that $\mu'\leq\rho^+$, hence
$$|\mu| -(|\bm c|+s-|F|)=|\mu'| \leq |\rho^+|,$$
which proves (ii). Finally, if we write $\rho =(\infty,\dots,\infty,\rho_l,\dots,\rho_n)$ with $\rho_l\neq\infty$, then Lemma~\ref{lem:claim} implies
\[ \mu_{l-1}\leq\cdots\leq\mu_1\leq\rho_l+1=\widetilde{\rho}_1=\cdots=\widetilde{\rho}_{l-1}.\]
Moreover, since $\mu'\leq\rho$, we get
\[ \mu_m\leq\mu'_m+1\leq\rho_m+1=\widetilde{\rho}_m\text{ for }l\leq m\leq n.\]
This shows that $\mu \leq \widetilde \rho$, proving (iii) and concluding our argument.
\end{proof}

We are now in the position to prove the main result of the section.

\begin{proof}[Proof of Theorem~\ref{thm:extremal}]
We know from Lemma~\ref{lem:maximal} that for each $\rho \in \Lambda^*_{\max}(I)$ we have
\[\Tor_{n-\ell(\rho)}(R/I,\kk)_{\langle \widetilde{\rho} \rangle}\neq 0.\]
To prove that every extremal pair for $R/I$ is of the form $(n-\ell(\rho),\widetilde{\rho})$, $\rho \in \Lambda^*_{\max}(I)$, it then suffices to check that for every pair $(j,\mu)$ with $\Tor_j(R/I,\kk)_{\langle \mu \rangle} \ne 0$, there exists $\rho \in \Lambda^*_{\max}(I)$ such that
\begin{equation}\label{eq:conds-jmu-smaller-rho}
\ \ j \leq n-\ell(\rho),\ \mu \leq \widetilde \rho \mbox{ and } |\mu|-j \leq |\widetilde \rho |-(n-\ell(\rho)).
\end{equation}

By Theorem~\ref{thm:main6}, $\Tor_j(R/I,\kk)_{\langle \mu \rangle} \ne 0$ implies that there is a $\bmc \leq \bf p(\mu)$ such that
$$\widetilde H_{s(\mu)-j-1+|\bmc|}(\Gamma^{\mu,\bm c}(I)) \ne 0.$$
We apply Lemma~\ref{lem:6.8} with $i=s(\mu)+|\bmc|-j$ to find $\rho \in \Lambda^*_{\max}(I)$ satisfying
\[j = |\bmc| +s(\mu)-i \leq n-\ell(\rho),\ |\mu|-j=|\mu|-(|\bmc| +s(\mu)-i) \leq |\rho^+|,\text{ and }\mu \leq \widetilde \rho.\]
Using the identity \eqref{eq:size-mu+-mutilde}, these conditions are precisely the ones from \eqref{eq:conds-jmu-smaller-rho}.

To conclude, we have to check that every pair $(n-\ell(\rho),\widetilde{\rho})$, with $\rho\in\Lambda^*_{\max}(I)$, is extremal. Equivalently, we have to show that if \eqref{eq:conds-jmu-smaller-rho} holds for $(j,\mu)=(n-\ell(\rho'),\widetilde{\rho}')$ with $\rho'\in\Lambda^*_{\max}(I)$, then $\rho=\rho'$. Indeed, in this case we can rewrite \eqref{eq:conds-jmu-smaller-rho} as
\[ l(\rho)\leq l(\rho'),\ \widetilde{\rho}'\leq\widetilde{\rho},\text{ and }l(\rho')-l(\rho)\leq |\widetilde{\rho}|-|\widetilde{\rho}'|,\]
which implies $\rho'\preccurlyeq\rho$. Since $\rho,\rho'\in\Lambda^*_{\max}(I)$, we must have $\rho'=\rho$, as desired.
\end{proof}

For a quick application of Theorem~\ref{thm:extremal}, recall that for a homogeneous ideal $I \subset R$,
the \defi{(Castelnuovo--Mumford) regularity} of $R/I$ is
$$\mathrm{reg}(R/I)= \max \{j-i\mid \Tor_i(R/I,\kk)_j \ne 0\}$$
and the \defi{projective dimension} of $R/I$ is
$$\mathrm{pd}(R/I)=\max\{ i \mid \Tor_i(R/I,\kk) \ne 0\}.$$
It follows that the extremal pairs in the Betti table of $R/I$ determine regularity and projective dimension, and we get the following.

\begin{corollary}\label{cor:reg} 
If $I \subset R$ is an $\sym_n$-invariant monomial ideal then
\begin{itemize}
\item[(i)] $\reg(R/I)=\max \{|\mu^+| \mid \mu \in \Lambda^*(I)\}$.
\item[(ii)] $\mathrm{pd}(R/I)=\max \{n-\ell(\mu) \mid \mu \in \Lambda^*(I)\}$.
\end{itemize}
\end{corollary}

\begin{proof}
For $\mu, \rho \in \Lambda^*(I)$, the condition
$\mu\preccurlyeq\rho$ implies that $\ell(\mu) \geq \ell(\rho)$ by Lemma \ref{lem:order}, and it implies $|\mu^+|\leq|\rho^+|$ by \eqref{eq:equiv-prec-on-LamI}.
It follows that the right side of the equations (i) and (ii) only depends on the maximal dual generators.
The conclusion then follows by combining Theorem \ref{thm:extremal} with \eqref{eq:size-mu+-mutilde}.
\end{proof}

\subsection{Extremal Betti numbers via $\Ext$ modules}
\label{subsec:extr-Betti-Ext}
The goal of this section is to explain how the extremal pairs and the extremal Betti numbers for $R/I$ can be recovered from the structure of the modules $\Ext^{\bullet}(R/I,R)$.
Using results from \cite{Ra},
which determine the structure of $\Ext^i(I,R)$ for any $\sym_n$-invariant monomial ideal,
we then give an alternative proof of Theorem~\ref{thm:extremal}, and we show how Corollary~\ref{cor:reg} is equivalent to \cite[(1.3)]{Ra}.



In analogy with the notion of extremal pair from the beginning of Section~\ref{sec:prim-dec-extr-Betti}, we say that a pair $(i,\lambda) \in [n] \times P_n$ is \defi{$\Ext$-extremal} if
\begin{itemize}
\item[(a)] $\Ext^i(R/I,R)_{\langle -\lambda \rangle} \ne 0$, and
\item[(b)] $\Ext^j(R/I,R)_{\langle -\mu \rangle}=0$ for all $j \geq i$ and $\mu \gneq \lambda$ with $|\mu|-j \geq |\lambda|-i$.
\end{itemize}
We first show that the $\Ext$-extremal pairs coincide with the extremal pairs, and moreover that the extremal Betti numbers can be computed via $\Ext$ modules as follows.

\begin{proposition}\label{prop:Ext-extremal}
 We have that $(i,\lambda)\in [n] \times P_n$ is extremal if and only if it is $\Ext$-extremal. Moreover, for an extremal pair $(i,\lambda)$ we have
 \begin{equation}\label{eq:bilam-from-Ext}
 \beta_{i,\lambda} = \dim_{\kk} \Ext^i(R/I,R)_{-\bm \lambda}.
 \end{equation}
\end{proposition}

Proposition~\ref{prop:Ext-extremal} is a natural extension of the corresponding result in the standard-graded setting (see for instance \cite[Proposition~1.1]{BCP}). When $I$ is a square-free monomial ideal (not necessarily $\sym_n$-invariant), it follows from \cite[Theorem~3.3]{Mustata} or \cite[Theorem~2.6]{Yanagawa} that the multigraded components of $\Ext^i(R/I,R)$ can be computed as Betti numbers of the Alexander dual $I^{\vee}$, in which case the equality \eqref{eq:bilam-from-Ext} is equivalent to the one proved in \cite[Theorem 2.8]{BCP}. 

To prove Proposition~\ref{prop:Ext-extremal}, we first establish a preliminary result. We write $F_{\bullet}$ for the minimal free resolution of $R/I$, so that
\[ F_i = \bigoplus_{\bmc\in\bb{Z}^n} R(-\bmc)^{\beta_{i,\bmc}}.\]
We let $F_{\bullet}^{\vee} = \Hom_R(F_{\bullet},R)$ be the dual of $F_{\bullet}$, so that $\Ext^i(R/I,R)$ is the $i$-th cohomology module of $F_{\bullet}^{\vee}$. We write $\pd^i:F_i^{\vee}\lra F_{i+1}^{\vee}$ for the differentials in $F_{\bullet}^{\vee}$. 

\begin{lemma}\label{lem:deg-Ext-classes}
 If $\Ext^i(R/I,R)_{-\bm \lambda}\neq 0$ then there exist $k\geq 0$ and $\mu\geq\lambda$, such that $|\mu|-|\lambda|\geq k$ and $F_{i+k}^{\bm \vee}$ has a minimal generator $m$ of degree $-\mu$, with $\pd^{i+k}(m)=0$. In particular, we have $\Ext^{i+k}(R/I,R)_{-\bm \mu}\neq 0$, and if $(i,\lambda)$ is $\Ext$-extremal, then every non-zero class in $\Ext^i(R/I,R)_{-\bm \lambda}$ is represented by a minimal generator of~$F_i^{\vee}$.
\end{lemma}

\begin{proof}
 By hypothesis, there exists $f_i\in\ker(\pd^i)$ with $\deg(f_i)=-\lambda$, representing a non-zero element of $\Ext^i(R/I,R)_{-\bm \lambda}$. For $j\geq 0$ and for as long as $0\neq f_{i+j}\in F_{i+j}^{\vee}$, we choose $m_{i+j}$ to be a minimal generator of $F_{i+j}^{\vee}$ of degree $-\mu^j\leq-\deg(f_{i+j})$, and set $f_{i+j+1}=\pd^{i+j}(m_{i+j})$. Since $F_{\bullet}^{\vee}$ is a finite complex, this process ends after finitely many steps with a minimal generator $m_{i+k}$ of $F_{i+k}^{\vee}$, satisfying $\pd^{i+k}(m_{i+k})=0$. We take $m=m_{i+k}$ and show that $\mu=\mu^k$ satisfies $\mu\geq\lambda$ and $|\mu|-|\lambda|\geq k$.
 
 We note that for each $j=0,\cdots,k-1$, $\deg(f_{i+j+1})=\deg(m_{i+j})=-\mu^j$, since the differential $\pd^{i+j}$ is degree-preserving. Moreover, since $\pd^{i+j}$ is minimal, it follows that $f_{i+j+1}$ is not a minimal generator of $F_{i+j+1}^{\vee}$, hence $-\mu^{j+1}<-\mu^j$ for $j=0,\cdots,k-1$. Starting with $\lambda=\deg(f_i)$, we get
 \[ \lambda\leq\mu^0<\mu^1<\cdots<\mu^k=\mu,\]
 which implies $\mu\geq\lambda$ and $|\mu|-|\lambda|\geq k$, as desired.
 
 Since $m$ is a minimal generator of $F_{i+k}^{\vee}$, $m$ is not a boundary, so it represents a non-zero element of $\Ext^{i+k}(R/I,R)_{-\bm \mu}$. If $(i,\lambda)$ is $\Ext$-extremal, this is only possible if $\mu=\lambda$ and $k=0$. If $f_i$ was not a minimal generator of $F_i^{\vee}$, then one can choose $m_i$ to be a minimal generator of $F_i^{\vee}$ of degree $-\mu^0<-\lambda$, and the construction above yields a pair $(i+k,\mu)$ with $\Ext^{i+k}(R/I,R)_{-\bm \mu}\neq 0$ and $|\mu|-|\lambda|\geq k$, contradicting the fact that $(i,\lambda)$ was $\Ext$-extremal, and concluding our proof.
\end{proof}

\begin{proof}[Proof of Proposition~\ref{prop:Ext-extremal}]
 Suppose first that $(i,\lambda)$ is an extremal pair, so $F_i$ has a minimal generator of degree $\lambda$, and there is no pair $(j,\mu)$ with $|\mu|-j\geq|\lambda|-i$ such that $F_j$ has a minimal generator of degree $\mu$. We show that every non-zero element $m\in F_i^{\vee}$ with $\deg(m)=-\lambda$ represents a non-zero class in $\Ext^i(R/I,R)_{-\bm \lambda}$. Indeed, we know that $m$ is not a boundary, since $\pd^{i-1}$ is minimal. Let $f=\pd^i(m)$, and suppose by contradiction that $f\neq 0$. Since $\deg(f)=\deg(m)=-\lambda$ and $\pd^i$ is minimal, there exists a minimal generator of $F_{i+1}^{\vee}$ of degree $-\mu<-\lambda$. This corresponds to a generator of $F_{i+1}$ of degree $\mu>\lambda$. This contradicts the fact that $(i,\lambda)$ was extremal, since it implies $|\mu|-(i+1)\geq|\lambda|-i$. It follows that $\pd^i(m)=0$, so $m$ represents a non-zero class in $\Ext^i(R/I,R)_{-\bm \lambda}$, as desired. By Lemma~\ref{lem:deg-Ext-classes}, every non-zero class in $\Ext^i(R/I,R)_{-\bm \lambda}$ arises in this way, so \eqref{eq:bilam-from-Ext} holds.

To show that $(i,\lambda)$ is $\Ext$-extremal, suppose by contradiction that there exists a pair $(j,\mu)$ with $\mu\gneq\lambda$, $|\mu|-j\geq |\lambda|-i$, and $\Ext^j(R/I,R)_{-\bm \mu}\neq 0$. Applying Lemma~\ref{lem:deg-Ext-classes} to $(j,\lambda)$, we can find $k\geq 0$ and $\delta\geq\mu$ with $|\delta|\geq|\mu|+k$, and such that $F_{j+k}^{\vee}$ has a minimal generator of degree $-\delta$. This shows that $F_{j+k}$ has a minimal generator of degree $\delta$, where $$|\delta|-(j+k)\geq|\mu|-j\geq |\lambda|-i,$$
contradicting the fact that $(i,\lambda)$ was extremal.

 Suppose now that $(i,\lambda)$ is $\Ext$-extremal. By Lemma~\ref{lem:deg-Ext-classes}, $F_i^{\vee}$ has a minimal generator of degree $-\lambda$, so $\Tor_i(R/I,\kk)_{\bm \lambda}\neq 0$. Suppose by contradiction that there exists a pair $(j,\mu)$ that satisfies $\mu\gneq\lambda$, $|\mu|-j\geq|\lambda|-i$, and $\Tor_j(R/I,\kk)_{\bm \mu}\neq 0$. We consider one such pair for which $j$ is maximal, and let $m$ denote a minimal generator of $F_j^{\vee}$ of degree $-\mu$. If $\pd^j(m)=0$ then $m$ represents a non-zero class in $\Ext^j(R/I,R)_{-\bm \mu}$ (since it is not a boundary), contradicting the fact that $(i,\lambda)$ was $\Ext$-extremal. If $\pd^j(m)=f\neq 0$ then there exists a minimal generator of $F_{j+1}^{\vee}$ of degree $-\delta<-\mu$. It follows that $\delta>\mu\gneq\lambda$,
 \[ |\delta|-(j+1) \geq |\mu|-j \geq |\lambda|-i,\]
 and $\Tor_{j+1}(R/I,\kk)_{\bm \delta}\neq 0$, so the pair $(j+1,\delta)$ contradicts the maximality of $(j,\mu)$, which concludes our proof.
\end{proof}

In order to apply Proposition~\ref{prop:Ext-extremal}, we recall some of the results and notation from \cite{Ra}. Let $\ul{z}=(z_1,\dots,z_n) \in P_n$ and $l\geq 0$, with $z_1=\cdots =z_{l+1}$.
We define the module $J_{\ul z,l}$ by (see \cite[(2.5)]{Ra})
\begin{equation}\label{eq:def-Jzl}
J_{\ul{z},l}= \langle \ul{z}\rangle_{\sym_n}
/ \langle \lambda \mid \lambda \geq z \mbox{ and } \lambda_i >z_i \mbox{ for some }i >l\rangle_{\sym_n}.
\end{equation}
\begin{itemize}
\item[(A)] In \cite[Corollary 2.13]{Ra} it was shown that $J_{\ul{z},l}$ is a Cohen--Macaulay $R$-module of dimension $l$ with an explicit description of $\Ext^{n-l}(J_{\ul{z},l},R)$.
\item[(B)] In \cite[Main theorem]{Ra} it was shown that for any $\sym_n$-invariant monomial ideal $I\subseteq R$, there is a finite set $\mc{Z}(I)\subset P_n\times\bb{Z}$ (with the notation in \cite[Definition~1.1]{Ra}, $\mc{Z}(I)=\mc{Z}(P(I))$) such that 
there exists a filtration of $R/I$ whose composition factors are the modules $J_{\ul{z},l}$ with $(\ul{z},l)\in\mc{Z}(I)$ and \[\Ext^i(R/I,R) \cong \bigoplus_{(\ul{z},l) \in \mc{Z}(I)} \Ext^i(J_{\ul{z},l},R).\]
\end{itemize}
We do not recall here the definition of $\mc{Z}(I)$, nor do we recall the description of $\Ext^{n-l}(J_{\ul{z},l},R)$, since they are somewhat technical. We only record the following property which follows from \cite[Corollary~2.13]{Ra}. If we define
\begin{equation}\label{eq:def-mu-z-l}
\mu(\ul{z},l) = (\infty^l,z_{l+1},\cdots,z_n) \in P_n^{\infty}
\end{equation}
then we have
\begin{itemize}
\item[(C)]
If $\ul{z}=(z_1,\dots,z_n)$ with $z_1=\cdots=z_p>z_{p+1}$ for some $p>l$, then
\[ \lambda = \ul{z}+(1^n) = (z_1+1,\cdots,z_n+1)=\widetilde{\mu(\ul{z},l)}\]
is the unique minimal element in the set
$\{\lambda \in P_n \mid \Ext^{n-l}(J_{\ul{z},l},R)_{-\bm \lambda} \ne 0\}$. Moreover $\dim_\kk \Ext^{n-l}(J_{\ul{z},l},R)_{-\bm \lambda}= {p-1 \choose l}$.
\end{itemize}

We explain a relation between the set $\mc{Z}(I)$ and dual generators.
Let
\begin{align*}
\Lambda(J_{\ul{z},l}) = \{\lambda\in P_n \mid (J_{\ul{z},l})_{\lambda}\neq 0\}  \overset{\eqref{eq:def-Jzl}}{=}\{\lambda\in P_n \mid \lambda \geq \ul{z},\ \lambda_i=z_i\text{ for }i\geq l+1\}.
\end{align*} 
Since $J_{\ul{z},l}$ are composition factors of $R/I$, we have a partition
$O(I)= \biguplus_{(\ul{z},l) \in \mc{Z}(I)} \Lambda(\ul{z},l).$
This allows us to write
\[ O(I) = \bigcup_{(\ul{z},l)\in\mc{Z}(I)} O_{\mu(\ul{z},l)},\]
but this representation of $O(I)$ is highly redundant. 

In analogy with $\Lambda^*(I)$, we define the set of \defi{dual pairs}
\begin{equation}\label{eq:def-dual-pairs}
\mc{Z}^*(I) = \{(\ul{z},l)\in\mc{Z}(I) \mid \mu(\ul{z},l)\not\leq\mu(\ul{y},u)\text{ for }(\ul{z},l)\neq(\ul{y},u)\in\mc{Z}(I)\}.
\end{equation}
We get an irredundant decomposition
\[ O(I) = \bigcup_{(\ul{z},l)\in\mc{Z}^*(I)} O_{\mu(\ul{z},l)},\]
and the formula \eqref{eq:def-mu-z-l} defines a bijection $\mc{Z}^*(I)\lra\Lambda^*(I)$.

\begin{example}\label{ex:Z-star-I}
 If $I=\langle(4,1,1),(5,2,0)\rangle_{\sym_3}$ as in Remark~\ref{rem:decomposition} then we have
{\small
\[ 
\begin{aligned}
\mc{Z}(I) =
\left\{ 
\begin{array}{l}
((0,0,0),1),((1,1,0),1),\\
((2,2,0),0),((3,2,0),0),((4,2,0),0),((3,3,0),0),((4,3,0),0),((4,4,0),0),\\
((1,1,1),0),((2,1,1),0),((3,1,1),0),((2,2,1),0),((3,2,1),0),((3,3,1),0),\\
((2,2,2),0),((3,2,2),0),((3,3,2),0),((3,3,3),0)
\end{array}
\right\}.
\end{aligned}
\]
}
The (significantly smaller) subset of dual pairs is 
\[ \mc{Z}^*(I) = \{((1,1,0),1),((4,4,0),0),((3,3,3),0)\},\]
corresponding to the dual generators of $I$
\[ \mu((1,1,0),1)=(\infty,1,0),\ \mu((4,4,0),0)=(4,4,0),\ \mu((3,3,3),0)=(3,3,3).\]
\end{example}

To give another perspective on $\mc{Z}^*(I)$, we introduce a partial order on $\mc{Z}(I)$ by
\[(\ul{z},l)\leq(\ul{y},u) \Longleftrightarrow l=u\text{ and }\ul{z}\leq\ul{y}.\]
We have that $(\ul{z},l),(\ul{y},u)$ are incomparable if $l\neq u$, and $(\ul{z},l)\leq(\ul{y},l)$ if and only if $\mu(\ul{z},l)\leq\mu(\ul{y},l)$. For the proof of the next result, we assume that the reader has some familiarity with \cite[Definition~1.1]{Ra} and its implications, such as \cite[Remark~2.3]{Ra} (in particular, we use some notation from \cite{Ra} which is not defined in this paper).

\begin{lemma}\label{lem:dual-pairs-maximal}
 We have that
 \[\mc{Z}^*(I) = \{\text{maximal elements of }\mc{Z}(I)\text{ with respect to }\leq\}.\]
\end{lemma}

\begin{proof}
 For the inclusion ``$\subseteq$", let $(\ul{z},l)\in\mc{Z}^*(I)$, and suppose by contradiction that there exists $(\ul{y},l)\in\mc{Z}(I)$ with $(\ul{y},l)>(\ul{z},l)$. This implies that $\mu(\ul{y},l)>\mu(\ul{z},l)$, contradicting \eqref{eq:def-dual-pairs}. 
 
 For the reverse inclusion ``$\supseteq$", let $(\ul{z},l)\in\mc{Z}(I)$ be maximal with respect to $\leq$, and suppose by contradiction that $(\ul{z},l)\not\in\mc{Z}^*(I)$. By \eqref{eq:def-dual-pairs}, there exists $(\ul{y},u)\in\mc{Z}(I)$ with $\mu(\ul{z},l)<\mu(\ul{y},u)$, which implies $l\leq u$. Moreover, if $l=u$ then $\ul{z}<\ul{y}$, contradicting the maximality of $(\ul{z},l)$. We thus have
 \begin{equation}\label{eq:conds-zl-yu}
 l<u\text{ and }z_i\leq y_i\text{ for }i\geq u+1.
 \end{equation}
 
 We write $c=z_1$ and $d=y_1$, and note that there exists $\ul{x}\in\mc{X}(I)$ such that $\ul{x}(c)\leq\ul{z}$ and $x'_{c+1}\leq l+1$. In particular, we must have $x_i\leq c$ for all $i>l+1$ (hence for $i\geq u+1$). Suppose first that $c\leq d$. The condition $\ul{x}(c)\leq\ul{z}$ implies that $x_i=\min(x_i,c)\leq z_i$ for $i\geq u+1$, which combined with \eqref{eq:conds-zl-yu} implies that $x_i\leq y_i$ for $i\geq u+1$. Since $y_i=d$ for $i\leq u+1$, this shows that $\ul{x}(d)\leq\ul{y}$. Since $(\ul{y},u)\in\mc{Z}(I)$, this forces $x'_{d+1}\geq u+1$, which contradicts \eqref{eq:conds-zl-yu} since it implies
 \[ u+1\leq x'_{d+1}\leq x'_{c+1}\leq l+1.\]
 Suppose now that $c>d$. We have that $z_i\leq y_i\leq d<c$ for $i\geq u+1$, which combined with $\ul{x}(c)\leq\ul{z}$ implies that $x_i\leq z_i$ for $i\geq u+1$. Using \eqref{eq:conds-zl-yu}, this shows that $x_i\leq y_i$ for $i\geq u+1$, hence $\ul{x}(d)\leq\ul{y}$. Since $(\ul{y},u)\in\mc{Z}(I)$, we must have $x'_{d+1}\geq u+1$, hence
 \[ x_{u+1}\geq d+1>y_{u+1}\geq z_{u+1}.\]
 Since $\ul{x}(c)\leq\ul{z}$, this forces $z_{u+1}=c$. The above inequality implies $d+1>z_{u+1}=c$, contradicting the fact that $d<c$ and concluding the proof.
\end{proof}

We can now explain how Corollary~\ref{cor:reg} is equivalent to the formulas \cite[(1.3)]{Ra}, which assert that
\[\reg(R/I) = \max\{|\ul{z}|+l \mid (\ul{z},l)\in\mc{Z}(I)\},\ 
\pdim(R/I) = \max\{n-l \mid (\ul{z},l)\in\mc{Z}(I)\}.\]
Indeed, it follows from Lemma~\ref{lem:dual-pairs-maximal} that for every $(\ul{z},l)\in\mc{Z}(I)$ there exists $(\ul{y},l)\in\mc{Z}(I)$ with $\ul{y}\geq\ul{z}$ (and hence $|\ul{y}|\geq|\ul{z}|$), which then implies that
\[\reg(R/I) = \max\{|\ul{z}|+l \mid (\ul{z},l)\in\mc{Z}^*(I)\},\ 
\pdim(R/I) = \max\{n-l \mid (\ul{z},l)\in\mc{Z}^*(I)\}.\]
Using the fact that if $\mu=\mu(\ul{z},l)$ then $|\mu^+|=|\ul{z}|+l$ and $\ell(\mu)=l$, this shows that Corollary~\ref{cor:reg} is equivalent to the above formulas. 

We conclude this section with an alternative proof of Theorem~\ref{thm:extremal}. To that end, we consider the ordering on $\mc{Z}^*(I)$ induced by the bijection with $\Lambda^*(I)$. We have 
\begin{equation}\label{eq:def-prec-on-ZI}
(\ul{z},l)\preccurlyeq(\ul{y},u)\Longleftrightarrow \ul{z}\leq\ul{y}\text{ and }|\ul{z}|+l\leq |\ul{y}|+u.
\end{equation}

\begin{proof}[Alternative proof of Theorem~\ref{thm:extremal}]
It follows from (B), (C) and  \eqref{eq:def-prec-on-ZI} that $(i,\lambda)$ is extremal if and only if there exists a maximal $(\ul{z},l)\in\mc{Z}^*(I)$ with respect to $\preccurlyeq$ such that $i=n-l$ and $\lambda=\widetilde{\mu(\ul{z},l)}$. Since this is equivalent to the fact that $\mu(\ul{z},l)\in\Lambda^*_{\max}(I)$, and since $\ell(\mu(\ul{z},l))=l$, this recovers the desired description of the extremal pairs.

Suppose that $(i,\lambda)$ is an extremal pair with
$(i,\lambda)=(n-l,\widetilde{\mu(\ul{z},l)})$.
If $p$ is the unique interger satisfying $z_1=\cdots=z_p>z_{p+1}$, then we have using Proposition~\ref{prop:Ext-extremal} and (C) that
\begin{equation}\label{eq:Ext-ilam-binomial}
\beta_{i,\lambda}=\dim_{\kk} \Ext^i(J_{\ul{z},l},R)_{-\bm \lambda} = {p-1\choose l}.
\end{equation}
If $\mu(\ul{z},l)=(\infty^{p_0},d_1^{p_1},\cdots,d_s^{p_s})$ with $\infty>d_1>\cdots>d_s\geq 0$ (as in Lemma~\ref{lem:maximal}), then we have $l=p_0$ and $p=p_0+p_1$. This means that \eqref{eq:Ext-ilam-binomial} agrees with the formula for the extremal Betti numbers from Lemma~\ref{lem:maximal}, concluding our proof.
\end{proof}

\section{Varying the number of variables}
\label{sec:vary-vars}
\ytableausetup{boxsize=0.3em}

In this section, we study how the multigraded Betti numbers of the ideals $I_m$ (defined in \eqref{eq:def-Im}) vary with $m$, when $f_1,\cdots,f_r$ are assumed to be monomials. This extends a result of the first author from \cite{Mu}, that gives a simple recipe to determine for all $m\geq n$ all the multidegrees $\mu \in P_m$ for which $\Tor_i(I_m,\kk)_{\langle \mu \rangle}$ is non-zero. The recipe requires knowing the set 
$$\{(i,\lambda) \in \{0,1,\dots,n-1\} \times P_n \mid \Tor_i(I_n,\kk)_{\langle \lambda \rangle} \ne 0\},$$
and is summarized in Theorem~\ref{thm:murai} below. The goal of this section is to explain how using Theorem~\ref{thm:main} we can determine not only which of the multigraded Betti numbers are non-zero, but to also compute them explicitly, and to describe the $\sym_m$-module structure of $\Tor_i(I_m,\kk)_{\langle \mu \rangle}$ for all $m \geq n$ and all $\mu\in P_m$. This is explained in Theorem~\ref{thm:varying}, but before going into details we make some preliminary conventions.

Throughout this section, we identify $(a_1,\dots,a_n)\in P_n$ and $(a_1,\dots,a_n,0^{m-n})\in P_m$ for $m\geq n$. By this identification, if $\lambda^1,\dots,\lambda^r \in P_n$, we can regard them as elements of $P_m$ with $m \geq n$, and consider the ideals
$$I_m=\langle \lambda^1,\dots,\lambda^r \rangle_{\sym_m} \subset \kk[x_1,\dots,x_m].$$
The (non-)vanishing of the multigraded Betti numbers of $I_m$ is characterized by the following theorem of the first author (see \cite[Theorem 3.2]{Mu}).

\begin{theorem}
\label{thm:murai}
Let $\lambda^1,\dots,\lambda^r \in P_n$ and let
$I_m=\langle \lambda^1,\dots,\lambda^r \rangle_{\sym_m}$ for $m \in \{n,n+1\}$.
For any $ 0 \leq i \leq n$ and $\mu=(\mu_1,\dots,\mu_n,\mu_{n+1}) \in P_{n+1}$, one has
\begin{itemize}
\item[(i)] if $\mu_{n+1}=0$ then $\Tor_i(I_{n+1},\kk)_\mu \ne 0 $ if and only if $\Tor_i(I_{n},\kk)_{(\mu_1,\dots,\mu_n)}\neq 0$,
\item[(ii)] if $0 < \mu_{n+1} <\mu_n$ then $\Tor_i(I_{n+1},\kk)_\mu =0$,
\item[(iii)] if $\mu_{n+1}=\mu_n$, then
$\Tor_i(I_{n+1},\kk)_\mu \ne 0$ if and only if $\Tor_{i-1}(I_n,\kk)_{(\mu_1,\dots,\mu_n)} \ne 0$.
\end{itemize}
\end{theorem}

To analyze the dimensions of the non-vanishing $\Tor$ groups in Theorem~\ref{thm:murai}, recall that for an $\sym_n$-invariant monomial ideal $I \subset \kk[x_1,\dots,x_n]$,
the numbers $\gamma_i^{\mu,\bm c}(I)$ determine all the (multigraded) Betti numbers of $I$. In order to determine the Betti numbers for $I_m$ when $m\geq n$, it is then enough to understand how each $\gamma_i^{\mu,\bm c}(I_m)$ changes when $m$ increases. This is explained by the following simple rule.

\begin{theorem}
\label{thm:varying}
Let $\lambda^1,\dots,\lambda^r \in P_n$ and $I_m=\langle \lambda^1,\dots,\lambda^r \rangle_{\sym_m}$ for $m \in\{n,n+1\}$.
Let $\mu=(\mu_1,\dots,\mu_n,\mu_{n+1}) \in P_{n+1}$, $s=s(\mu)$ and $\bm c \leq \bf p(\mu)$. For every $ 0\leq i \leq n$, we have
\begin{itemize}
\item[(i)] if $\mu_{n+1}=0$, then $\gamma_i^{\mu,\bm c}(I_{n+1})=\gamma_i^{(\mu_1,\dots,\mu_n),\bm c}(I_n)$,
\item[(ii)] if $\mu_{n+1}>0$, then
$$\gamma_i^{\mu,\bm c}(I_{n+1})=
\begin{cases}
0& \mbox{ if }c_s=0;\\
\gamma_i^{(\mu_1,\dots,\mu_n),\bm c-\ee_s}& \mbox{ if } c_s>0.
\end{cases}
$$
\end{itemize}
\end{theorem}

We note that Theorem \ref{thm:murai} can be recovered from Theorem \ref{thm:varying} using Theorem~\ref{thm:main}. If $\mu_{n+1}> \mu_n$ then $c_s=0$, so Theorem \ref{thm:murai}(ii) follows from Theorem \ref{thm:varying}(ii).

\begin{proof}
We first observe that for any $\rho=(\rho_1,\dots,\rho_n,\rho_{n+1}) \in P_{n+1}$, one has
\begin{align}
\label{7-1}
\rho \in P(I_{n+1}) \Longleftrightarrow (\rho_1,\dots,\rho_n) \in P(I_n)
\end{align}
since both conditions in \eqref{7-1} are equivalent to the condition that $\rho \geq \lambda_k$ for some~$k$.
Throughout the proof we will write $\hat \mu=(\mu_1,\dots,\mu_n)$.

(i) If $\mu_{n+1}=0$ then it follows from \eqref{eq:def-Del-mu-c-I} that
$\Delta^{\mu,\bm c}(I_{n+1})=\Delta^{\hat \mu,\bm c}(I_n)$, and therefore $\gamma_i^{\mu,\bm c}(I_{n+1})=\gamma_i^{\hat \mu, \bm c}(I_n)$ for all $i$, proving (i).

(ii) Suppose now that $\mu_{n+1} >0$.
We first consider the case when $c_s=0$, and show that $\Delta^{\mu,\bm c}(I_{n+1})$ is a cone with apex $s$. Indeed, suppose that $c_s=0$, consider any face $F \in \Delta^{\mu,\bm c}(I_{n+1})$ with $s \not \in F$, and let $\mu \setminus (\bm c + \ee_F) =(\mu_1',\dots,\mu_{n+1}')$.
We have 
\[\mu \setminus (\bm c + \ee_{F \cup \{s\}}) = \mu \setminus (\bm c +\ee_F) -\eee_{n+1} = (\mu'_1,\cdots,\mu'_n,\mu'_{n+1}-1)\] 
since $c_s=0$, so applying \eqref{7-1} twice we obtain
$$\mu \setminus (\bm c + \ee_{F \cup \{s\}}) \in P(I_{n+1}) \Longleftrightarrow (\mu_1',\dots,\mu_n') \in P(I_n) \Longleftrightarrow \mu \setminus (\bm c+ \ee_F) \in P(I_{n+1}).$$
This proves that $F \cup \{s \} \in \Delta^{\mu,\bm c}(I_{n+1})$ and therefore $\Delta^{\mu,\bm c}(I_{n+1})$ is a cone with apex~$s$. We get $\widetilde H_i(\Delta^{\mu,\bm c}(I_{n+1}))=0$, and hence $\gamma_i^{\mu,\bm c}(I_{n+1})=0$, for all $i$.

To conclude, we consider the case when $c_s > 0$. Since $c_s\leq p_s-1$ by hypothesis, we have $p_s\geq 2$, hence $\mu_n=\mu_{n+1}$, and therefore ${\bf p}(\hat \mu)={\bf p}(\mu)-\ee_{s}$. We claim that
\begin{equation}\label{eq:DIn+1=DIn}
\Delta^{\mu,\bm c}(I_{n+1})=\Delta^{\hat \mu, \bm c-\ee_s}(I_n).
\end{equation}
To prove \eqref{eq:DIn+1=DIn}, consider any subset $F \subset [s]$ and write $\mu \setminus (\bm c + \ee_F)=(\mu_1',\dots,\mu_{n+1}')$ as before. Since ${\bf p}(\hat \mu)={\bf p}(\mu) -\ee_s$, it follows that
\begin{align}
\label{7-2}
\hat \mu \setminus (\bm c -\ee_s+ \ee_F)=(\mu_1',\dots,\mu_n').
\end{align}
Using \eqref{7-1}, we have
$$\mu \setminus (\bm c + \ee_F) \in P(I_{n+1}) \Longleftrightarrow (\mu_1',\dots,\mu_n') \in P(I_n),$$
which combined with \eqref{7-2} implies that $F \in \Delta^{\mu,\bm c}(I_{n+1})$ if and only if $F \in \Delta^{\hat \mu,\bm c-\ee_s}(I_n)$. This proves \eqref{eq:DIn+1=DIn}, showing that $\gamma_i^{\mu,\bm c}(I_{n+1})=\gamma_i^{\hat{\mu},\bm c-\ee_s}$ for all $i$, as desired.
\end{proof}

\begin{example}
Let $I_m=\langle (5,1),(2,2) \rangle_{\sym_m}$ for $m \geq 2$. If we apply Theorem~\ref{thm:main} to~$I_2$, we see that the only numbers $\gamma_i^{\mu,\bm c}$ that are non-zero are
\begin{equation}\label{eq:gams-m=2}
\gamma_0^{(2,2),(0)}(I_2) = \gamma_0^{(5,1),(0)}(I_2)= \gamma_1^{(5,2),(0,0)}(I_2)=1.
\end{equation}
In particular, we have
$$\Tor_0(I_2,\kk) \cong \Tor_0(I_2,\kk)_{\langle (2,2)\rangle} \oplus \Tor_0(I_2,\kk)_{\langle (5,2)\rangle} \cong S^{\ydiagram 2} \oplus \mathcal S^{(\ydiagram 1, \ydiagram 1)}$$
and
$$\Tor_1(I_2,\kk) \cong \Tor_1(I_2,\kk)_{\langle (5,2) \rangle } \cong \mathcal S^{(\ydiagram 1, \ydiagram 1)}.$$

Based on \eqref{eq:gams-m=2}, Theorem \ref{thm:varying} gives a recipe to compute all the numbers $\gamma_i^{\mu,\bm c}$ for all the ideals $I_m$ with $m \geq 2$. For instance, when $m=4$ we have
\begin{align*}
&
\gamma_0^{(2,2,0,0),(0)}(I_4)=
\gamma_0^{(2,2,2,0),(1)}(I_4)=
\gamma_0^{(2,2,2,2),(2)}(I_4)=1,\\
&\gamma_0^{(5,1,0,0),(0,0)}(I_4)=
\gamma_0^{(5,1,1,0),(0,1)}(I_4)=
\gamma_0^{(5,1,1,1),(0,2)}(I_4)=1,\\
&\gamma_1^{(5,2,0,0),(0,0)}(I_4)=
\gamma_1^{(5,2,2,0),(0,1)}(I_4)=
\gamma_1^{(5,2,2,2),(0,2)}(I_4)=1,
\end{align*}
and these are all the non-zero numbers $\gamma_i^{\mu,\bm c}$ for $I_4$.
By Theorem~\ref{thm:main}, the $\sym_4$-module structure of $\Tor_i(I_4,\kk)$ is then computed as follows.
$$
\Tor_0(I,\kk) \cong
\Tor_0(I,\kk)_{\langle(2,2,0,0)\rangle} \oplus
\Tor_0(I,\kk)_{\langle(5,1,0,0)\rangle} 
\cong \mathcal S^{(\ydiagram 2, \ydiagram 2)} \oplus \mathcal S^{(\ydiagram 1, \ydiagram 1, \ydiagram 2)},$$
\begin{align*}
\Tor_1(I,\kk) &\cong
\Tor_1(I,\kk)_{\langle(2,2,2,0)\rangle} \oplus
\Tor_1(I,\kk)_{\langle(5,1,1,0)\rangle} \oplus
\Tor_1(I,\kk)_{\langle(5,2,0,0)\rangle}\\ 
& \cong \mathcal S^{(\ydiagram {2,1}, \ydiagram 1)} \oplus \mathcal S^{(\ydiagram 1, \ydiagram {1,1}, \ydiagram 1)}\oplus \mathcal S^{(\ydiagram 1, \ydiagram 1, \ydiagram 2)},
\end{align*}
\begin{align*}
\Tor_2(I,\kk) &\cong
\Tor_2(I,\kk)_{\langle(2,2,2,2)\rangle} \oplus
\Tor_2(I,\kk)_{\langle(5,1,1,1)\rangle} \oplus
\Tor_2(I,\kk)_{\langle(5,2,2,0)\rangle}\\ 
& \cong S^{\ydiagram {2,1,1}} \oplus \mathcal S^{(\ydiagram 1, \ydiagram {1,1,1})}\oplus \mathcal S^{(\ydiagram 1, \ydiagram {1,1}, \ydiagram 1)},
\end{align*}
and
$$
\Tor_3(I_4,\kk) \cong \Tor_3(I_4,\kk)_{\langle (5,2,2,2)\rangle } \cong \mathcal S^{(\ydiagram 1, \ydiagram {1,1,1})}.$$
By computing the dimensions of the relevant $\sym_4$-representations, we obtain the Betti tables of $I_2$ (left) and $I_4$ (right) below.

\[
\begin{array}{c}
\begin{matrix}
&0&1\\
\text{total:}&3&2\\
\text{4:}&1&\text{.}\\
\text{5:}&\text{.}&\text{.}\\
\text{6:}&2&2\\
\end{matrix}
\end{array}
\hspace{60pt}
\begin{array}{c}
\begin{matrix}
&0&1&2&3\\
\text{total:}&18&32&19&4\\
\text{4:}&6&\text{.}&\text{.}&\text{.}\\
\text{5:}&\text{.}&8&\text{.}&\text{.}\\
\text{6:}&12&24&7 &\text{.}\\
\text{7:}&\text{.}&\text{.}&12&\text{.}\\
\text{8:}&\text{.}&\text{.}&\text{.}&4
\end{matrix}
\end{array}
\]
\end{example}

Theorem \ref{thm:varying}(ii) tells that the representation of $\Tor_i(I_\ell,\kk)$ and that of $\Tor_{i+1}(I_{\ell+1},\kk)$ are related when $I_m$ is generated by monomials.
We expect that a similar phenomenon occurs even when $I_m$ is not generated by monomials,
and end this paper with the following question, which is inspired from our result and a result given in \cite{SY}

\begin{question}
Let $I_m$ be as in \eqref{eq:def-Im}
and let $i$ be a sufficiently large integer.
Suppose $\mathrm{char}(\kk)=0$.
Is is true that, for all $\ell > i$,
if $S^{(\lambda_1,\dots,\lambda_r)}$ is a summand of $\Tor_i(I_\ell,\kk)$, then $S^{(\lambda_1,\dots,\lambda_r,1)}$ is a summand of $\Tor_i(I_{\ell+1},\kk)$?
\end{question}

\section*{Acknowledgements}

The second author would like to thank Eric Ramos and Steven Sam for helpful conversations regarding this project. Experiments with the computer algebra software Macaulay2 \cite{M2} have provided numerous valuable insights. 
The first author acknowledges the support Waseda University Grant Research Base Creation 2020C-147.
The second author acknowledges the support of the National Science Foundation Grant No.~1901886.

\bibliographystyle{plain}
\bibliography{Hochster}


\end{document}